\documentclass[reqno]{siamart171218}
\usepackage{amsmath, amscd, amssymb}
\usepackage{bm, mathrsfs}
\usepackage{mathtools}
\usepackage{graphicx,epsfig,epstopdf}
\usepackage{enumerate}
\usepackage{subfig}
\usepackage{url}
\usepackage{array}
\newcolumntype{P}[1]{>{\centering\arraybackslash}p{#1}}
\newcolumntype{M}[1]{>{\centering\arraybackslash}m{#1}}
\usepackage{algorithmic}
\usepackage[ddmmyy,hhmmss]{datetime}
\usepackage{soul}
\newcommand{\mathcolorbox}[2]{\colorbox{#1}{$\displaystyle #2$}}

\newsiamremark{remark}{Remark}
\newsiamremark{example}{Example}

\DeclareMathOperator{\diag}{diag}
\DeclareMathOperator{\spn}{span}	

\DeclareMathOperator{\spnS}{\spn^{\mathbb{S}}}
\DeclareMathOperator{\spec}{spec}	

\newcommand{\tr}[1]{\text{trace}\left(#1\right)}

\renewcommand{\d}{\,\mathrm{d}}		

\newcommand{\fold}[1]{{\tt fold}{\left(#1\right)}}
\newcommand{\unfold}[1]{{\tt unfold}{\left(#1\right)}}
\newcommand{\bcirc}[1]{{\tt bcirc}{\left(#1\right)}}

\usepackage{xspace}		

\newcommand{\fomfom}{(FOM)\textsuperscript{2}\xspace}

\newcommand{\bfomfom}{B\fomfom}

\newcommand{\spC}{\mathbb{C}}				

\newcommand{\spCnn}{\spC^{n \times n}}

\newcommand{\spCNN}{\spC^{np \times np}}
\newcommand{\spCNn}{\spC^{np \times n}}
\newcommand{\spCNnm}{\spC^{np \times nm}}

\newcommand{\spCtensor}{\spC^{n_1 \times n_2 \times n_3}}
\newcommand{\spK}{\mathscr{K}}				
\newcommand{\spR}{\mathbb{R}}				
\newcommand{\spS}{\mathbb{S}}

\newcommand{\spSmm}{\spS^{m \times m}}
\newcommand{\spKS}{\spK^{\spS}}        

\newcommand{\vd}{\bm{d}}
\newcommand{\ve}{\bm{e}}
\newcommand{\vehat}{\widehat{\ve}}

\newcommand{\vB}{\bm{B}}

\newcommand{\vDtil}{\widetilde{\bm{D}}}
\newcommand{\vE}{\bm{E}}

\newcommand{\vF}{\bm{F}}

\newcommand{\vQ}{\bm{Q}}

\newcommand{\vV}{\bm{V}}

\newcommand{\vW}{\bm{W}}
\newcommand{\vX}{\bm{X}}

\newcommand{\vY}{\bm{Y}}

\newcommand{\vZ}{\bm{Z}}

\newcommand{\Beta}{B}

\renewcommand{\AA}{\mathcal{A}}
\newcommand{\BB}{\mathcal{B}}
\newcommand{\CC}{\mathcal{C}}
\newcommand{\DD}{\mathcal{D}}

\newcommand{\HH}{\mathcal{H}}

\newcommand{\HHunder}{\underline{\HH}}
\newcommand{\II}{\mathcal{I}}

\newcommand{\XX}{\mathcal{X}}
\newcommand{\XXarrow}{\vec{\XX}}

\newcommand{\bVV}{\bm{\mathcal{V}}}

\newcommand{\inv}{{-1}}

\newcommand{\one}{{(1)}}

\renewcommand{\k}{{(k)}}
\newcommand{\m}{{$(m)$}\xspace}

\newcommand{\kmore}{{(k+1)}}

\newcommand{\errfunc}{\varDelta}	
\newcommand{\errapprox}{\vDtil}	

\newcommand{\fAABB}{f(\AA)*\BB}

\newcommand{\fAB}{f(A)\vB}

\newcommand{\norm}[1]{\left\lVert#1\right\rVert}

\newcommand{\normF}[1]{\norm{#1}_{\text{F}}}

\makeatletter
\newsavebox{\@brx}
\newcommand{\llangle}[1][]{\savebox{\@brx}{\(\m@th{#1\langle}\)}%
	\mathopen{\copy\@brx\kern-0.5\wd\@brx\usebox{\@brx}}}
\newcommand{\rrangle}[1][]{\savebox{\@brx}{\(\m@th{#1\rangle}\)}%
	\mathclose{\copy\@brx\kern-0.5\wd\@brx\usebox{\@brx}}}
\makeatother

\newcommand{\IPS}[2]{{\llangle #1,#2 \rrangle}_{\spS}}

\newcommand{\IPSnoarg}{\IPS{\cdot}{\cdot}}

\newcommand{\N}{{N}}

\newcommand{\MATLAB}{\textsc{Matlab}\xspace}
\newcommand{\bigO}{\mathcal{O}}
\newcommand{\cost}{\mathcal{C}}

\usepackage{xpatch}
\makeatletter
\xpatchcmd{\listoftodos}{\@starttoc}{\mbox{}\par\@starttoc}{}{}
\makeatother
\newcommand{\tol}{\texttt{tol}\xspace}
\newcommand{\cf}{cf.\ }

\definecolor{SunsetNavy}{HTML}{160A47}
\definecolor{SunsetNavyLL}{HTML}{D0CEDA}

\definecolor{SunsetPurpleD}{HTML}{430254}
\definecolor{SunsetPurple}{HTML}{600479}
\definecolor{SunsetPurpleL}{HTML}{9F68AE}
\definecolor{SunsetPurpleLL}{HTML}{CFB3D6}

\definecolor{SunsetPinkD}{HTML}{6C0A42}
\definecolor{SunsetPink}{HTML}{9B0F5F}
\definecolor{SunsetPinkL}{HTML}{C36F9F}
\definecolor{SunsetPinkLL}{HTML}{E1B7CF}

\definecolor{SunsetRedD}{HTML}{8C121A}
\definecolor{SunsetRed}{HTML}{C81B26}
\definecolor{SunsetRedL}{HTML}{DE767C}
\definecolor{SunsetRedLL}{HTML}{EEBABD}

\definecolor{SunsetOrangeD}{HTML}{A84815}
\definecolor{SunsetOrange}{HTML}{F1671F}
\definecolor{SunsetOrangeL}{HTML}{F59462}
\definecolor{SunsetOrangeLL}{HTML}{F9C2A5}

\newcommand{\TheTitle}{%
  The tensor t-function: a definition for functions of third-order tensors
}

\newcommand{\TheShortTitle}{Tensor t-function}

\newcommand{\TheName}{Kathryn Lund}

\newcommand{\TheAddress}{%
  Charles University, Prague, Czech Republic
  (\email{kathryn.lund@karlin.mff.cuni.cz}).
}

\newcommand{\TheFunding}{%
  This work was supported in part by the U.S.\ National Science Foundation under
  grant DMS-1418882, the U.S.\ Department of Energy under grant DE-SC 0016578, and the Charles University PRIMUS grant, project no. PRIMUS/19/SCI/11
  \@.
}

\author{\TheName\thanks{\TheAddress}}
\title{{\TheTitle}\thanks{\TheFunding}}
\headers{\TheShortTitle}{\TheName}
\ifpdf%
\hypersetup{%
  pdftitle={\TheTitle},
  pdfauthor={\TheName}
}
\fi

\begin{document}
\maketitle

\begin{abstract}
	A definition for functions of multidimensional arrays is presented.  The definition is valid for third-order tensors in the tensor t-product formalism, which regards third-order tensors as block circulant matrices.  The tensor function definition is shown to have similar properties as standard matrix function definitions in fundamental scenarios.  To demonstrate the definition's potential in applications, the notion of network communicability is generalized to third-order tensors and computed for a small-scale example via block Krylov subspace methods for matrix functions.  A complexity analysis for these methods in the context of tensors is also provided.
\end{abstract}



\section{Introduction}\label{sec:intro}
Functions of matrices-- that is, $f(A)$, where $f$ is a scalar function, and $A$ a square matrix-- have applications in a number of fields.  They emerge as measures of centrality and communicability in networks \cite{EstradaHigham2010} and as exponential integrators in differential equations \cite{HochbruckOstermann2010}.  As high-dimensional analogues of matrices, tensors also play crucial roles in network analysis \cite{Cichocki2014} and multidimensional differential equations \cite{Khoromskij2015}.  A variety of decompositions and algorithms have been developed over the years to extract and understand properties of tensors \cite{KoldaBader2008}.  A natural question is whether the notion of functions of tensors, defined in analogy to functions of matrices as a scalar function taking a tensor $\AA$ as its argument, could prove to be yet another useful tool for studying multidimensional data.

Unfortunately, the definition of such a notion is not nearly as straightforward for tensors as it is for matrices.  For matrices, the definitions of integration, polynomials, eigendecompositions (ED), and singular value decompositions (SVD) are unique and well established throughout linear algebra, and all of these notions serve as building blocks for definitions of matrix functions, reducing to the same object under reasonable circumstances \cite{ArrigoBenziFenu2016, Higham2008}.  With the abundance of generalizations for the ED and SVD in the world of tensors \cite{deLathauwerdeMoorVandewalle2000, KilmerBramanHaoHoover2013, KoldaBader2008, KoldaMayo2011, Lim2005, NgQiZhou2009, Qi2005, Qi2007}, there is no guarantee that a tensor function definition based on one type of decomposition is equivalent to a definition based on another type.

With that in mind, we embark on a first look at defining functions of tensors via a simple paradigm, the tensor t-product formalism \cite{Braman2010, KilmerBramanHaoHoover2013, KilmerMartin2011}.  This paradigm effectively regards third-order tensors as block circulant matrices, and we exploit this fact to develop a tensor function definition $\fAABB$ that effectively reduces to $\fAB$, i.e., the action of a matrix function on a block vector.  This fact, along with properties of block circulant matrices, also allows for the transfer of important matrix function properties to our tensor function definition.  Although the paradigm we consider is limited in applicability to only third-order tensors, it serves as an important first step in exploring tensor function definitions.

Thanks to the equivalence of our tensor function definition with the $\fAB$ problem, we have a number of options for computing $\fAABB$, such as \cite{AlMohyHigham2011, FrommerLundSzyld2017, FrommerLundSzyld2019, SimonciniLopez2006}.  We focus on adapting the block Krylov subspace methods (KSMs) from \cite{FrommerLundSzyld2017}, and in light of the so-called ``curse of dimensionality," we also present a computational complexity analysis for this algorithm in the tensor function context.  We also propose modifications to the algorithm based on the discrete Fourier transform, which were shown in \cite{KilmerBramanHaoHoover2013} to increase computational efficiency for the tensor t-product.

This report proceeds as follows.  We recapitulate matrix function definitions and properties in Section~\ref{sec:def_matrix_functions}.  Section~\ref{sec:def_tensor_functions} restates the tensor t-product framework and poses a definition for the {\em tensor t-function}, a new definition for a tensor function within this framework.  We also present statements and proofs of t-function properties in analogy to the core properties of matrix functions.  A possible application for the tensor t-exponential as a generalized communicability measure is discussed in Section~\ref{sec:network_analysis}.  In Section~\ref{sec:methods_compute}, we show how block KSMs for matrix functions can be used to compute the tensor t-function, and demonstrate the efficacy of these methods for the tensor t-exponential.  We make concluding remarks in Section~\ref{sec:conclusions}.

Before proceeding, we make a brief comment on syntax and disambiguation: the phrase ``tensor function" already has an established meaning in physics; see, e.g., \cite{Betten2008, Boehler1987, Zheng2009}.  The most precise phrase for our object of interest would be ``a function of a multidimensional array," in analogy to ``a function of a matrix."  However, since combinations of prepositional phrases can be cumbersome in English, we risk compounding literature searches by resorting to the term ``tensor function."

\subsection{Definitions of matrix functions} \label{sec:def_matrix_functions}
Following \cite{FrommerSimoncini2008a, Higham2008}, we concern ourselves with the three main matrix function definitions, based on the Jordan canonical form, Hermite interpolating polynomials, and the Cauchy-Stieltjes integral form.  In each case, the validity of the definition boils down to the differentiability of $f$ on the spectrum of $A$.  When $f$ is analytic on the spectrum of $A$, all the definitions are equivalent, and we can switch between them freely.

Let $A \in \spCnn$ be a matrix with spectrum $\spec(A) := \{\lambda_j\}_{j=1}^{N}$, where $N \leq n$ and the $\lambda_j$ are distinct.  An $m \times m$ Jordan block $J_m(\lambda)$ of an eigenvalue $\lambda$ has the form
\[
J_m(\lambda) =
\begin{bmatrix}
\lambda & 1 &  &  \\
& \lambda & \ddots &  \\
&	& \ddots & 1 \\
&	&	& \lambda 
\end{bmatrix}
\in \spC^{m \times m}.
\]
Suppose that $A$ has Jordan canonical form
\begin{equation} \label{eq:Jordan}
A = X J X^\inv = X^\inv \diag(J_{m_1}(\lambda_{j_1}), \ldots, J_{m_p}(\lambda_{j_\ell})) X,
\end{equation}
with $p$ blocks of sizes $m_i$ such that $\sum_{i=1}^{p} m_i = n$, and where the values \linebreak$\{\lambda_{j_k}\}_{k=1}^\ell \in \spec(A)$.  Note that eigenvalues may be repeated in the sequence $\{\lambda_{j_k}\}_{k=1}^\ell$.  Let $n_j$ denote the {\em index} of $\lambda_j$, or the size of the largest Jordan block associated to~$\lambda_j$.

A function is {\em defined on the spectrum of $A$} if all the following values exist:
\[
f^{(k)}(\lambda_j), \qquad k = 0, \ldots, n_j-1, \qquad j = 1, \ldots, N.
\]

\begin{definition} \label{def:mat_func_jordan}
	Suppose $A \in \spCnn$ has Jordan form~\eqref{eq:Jordan} and that $f$ is defined on the spectrum of $A$.  Then we define
	\[
	f(A) := X f(J) X^\inv,
	\]
	where $f(J) := \diag(f(J_{m_1}(\lambda_{j_1})), \ldots, f(J_{m_p}(\lambda_{j_\ell})))$, and
	\[
	f(J_{m_i}(\lambda_{j_k}))
	:=
	\begin{bmatrix}
	f(\lambda_{j_k}) & f'(\lambda_{j_k}) & \frac{f''(\lambda_{j_k})}{2!}& \dots		& \frac{f^{(n_{j_k}-1)}(\lambda_{j_k})}{(n_{j_k}-1)!}	\\
	0			 & f(\lambda_{j_k})	 & f'(\lambda_{j_k}) 			& \dots		& \vdots	\\
	\vdots		 & \ddots		 & \ddots			  		& \ddots	& \frac{f''(\lambda_{j_k})}{2!}	\\
	\vdots		 &				 & \ddots 			  		& \ddots    & f'(\lambda_{j_k})		\\
	0			 & \dots		 & \dots 			  		& 0			& f(\lambda_{j_k})
	\end{bmatrix} \in \spC^{m_i \times m_i}
	\]
\end{definition}
Note that when $A$ is diagonalizable with $\spec(A) = \{\lambda_j\}_{j=1}^n$ (possibly no longer distinct), Definition~\ref{def:mat_func_jordan} reduces to
\[
f(A) = X \diag(f(\lambda_1), \ldots, f(\lambda_n)) X^\inv.
\]

Matrix powers are well defined, so a scalar polynomial evaluated on a matrix is naturally defined.  Given $p(z) = \sum_{k=0}^{m} z^k c_k$, for some $\{c_k\}_{k=1}^m \subset \spC$, we have that $p(A) := \sum_{k=1}^m A^k c_k$.  Based on this, we can define non-polynomial functions of matrices by using again derivatives as we did in Definition~\ref{def:mat_func_jordan}.
\begin{definition} \label{def:mat_func_Hermite}
	Suppose that $f$ is defined on $\spec(A)$, and let $p$ with \linebreak$\deg p \leq \sum_{j=1}^{N} n_j$ be the unique Hermite interpolating polynomial satisfying
	\[
	p^{(k)}(\lambda_j) = f^{(k)}(\lambda_j), \mbox{ for all } k = 0, \ldots, n_{j-1}, \quad j = 1, \ldots, N.
	\]
	We then define $f(A) := p(A)$.
\end{definition}

\begin{theorem}[Theorem~1.3 from \cite{Higham2008}] \label{thm:polynomial_equivalence}
	For polynomials $p$ and $q$ and $A \in \spCnn$, $p(A) = q(A)$ if and only if $p$ and $q$ take the same values on the spectrum of $A$.
\end{theorem}
The proof follows by noting that the minimal polynomial of $A$-- i.e., the polynomial $\psi$ of least degree such that $\psi(A) = 0$-- divides $p-q$, and consequences thereof.

Crucial for our methods and analysis is the Cauchy-Stieltjes integral definition.
\begin{definition} \label{def:mat_func_integral}
	Let $\mathbb{D} \subset \spC$ be a region, and suppose that $f: \mathbb{D} \to \spC$ is analytic with integral representation
	\begin{equation} \label{def:Cauchy-Stieltjes_func}
	f(z) = \int_{\Gamma} \frac{g(t)}{t-z} \d t, \quad  z \in \mathbb{D},
	\end{equation}
	with a path $\Gamma \subset \spC \setminus \mathbb{D}$ and function $g: \Gamma \to \spC$.  Further suppose that the spectrum of $A$ is contained in $\spC \setminus \mathbb{D}$.  Then we define
	\[
	f(A) := \int_{\Gamma} g(t)(tI - A)^\inv \d t.
	\]
\end{definition}
When $f$ is analytic, $g = \frac{1}{2 \pi i}f$, and $\Gamma$ is a contour enclosing the spectrum of $A$, then Definition~\ref{def:Cauchy-Stieltjes_func} reduces to the usual Cauchy integral definition.

Various matrix function properties will prove useful throughout our analysis.  Their proofs follow by examining the polynomial and Jordan form definitions of matrix functions.
\begin{theorem}[Theorem~1.13 in \cite{Higham2008}] \label{thm:mat_func_props}
	Let $A \in \spCnn$ and let $f$ be defined on the spectrum of $A$.  Then
	\begin{enumerate}[(i)]
		\item $f(A) A = A f(A)$;
		\item $f(A^T) = f(A)^T$;
		\item $f(X A X^\inv) = X f(A) X^\inv$; and
		\item $f(\lambda) \in \spec(f(A))$ for all $\lambda \in \spec(A)$.
	\end{enumerate}
\end{theorem}

\section{A definition for tensor functions} \label{sec:def_tensor_functions}
\newcommand{\imsize}{.13\textwidth}
\begin{figure}[htbp!]
	\begin{center}
		\begin{tabular}{c c c c c c}
			{\resizebox{\imsize}{!}{
					\includegraphics[width=\imsize]{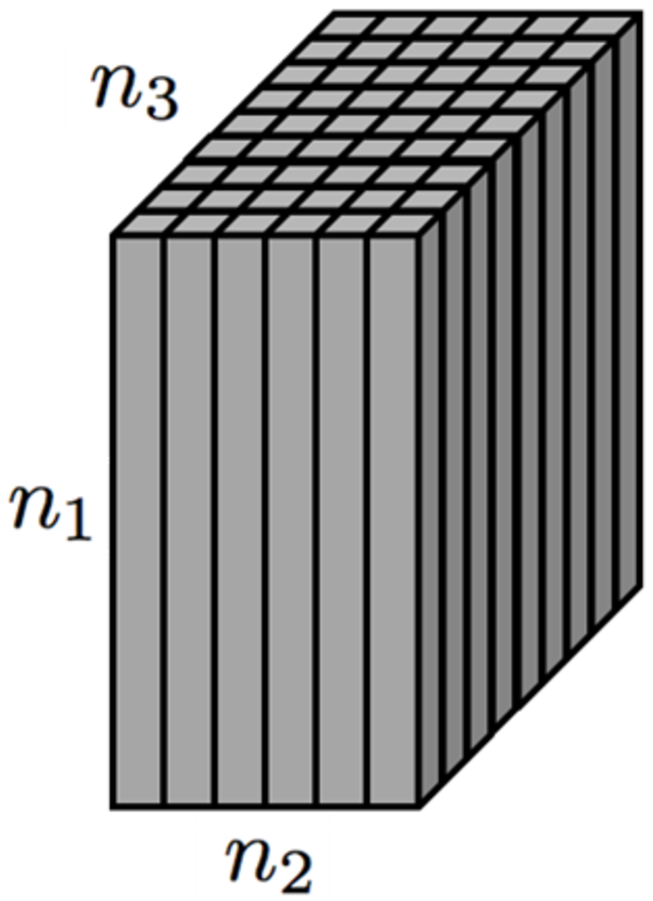}}} &
			
			{\resizebox{\imsize}{!}{
					\includegraphics[width=\imsize]{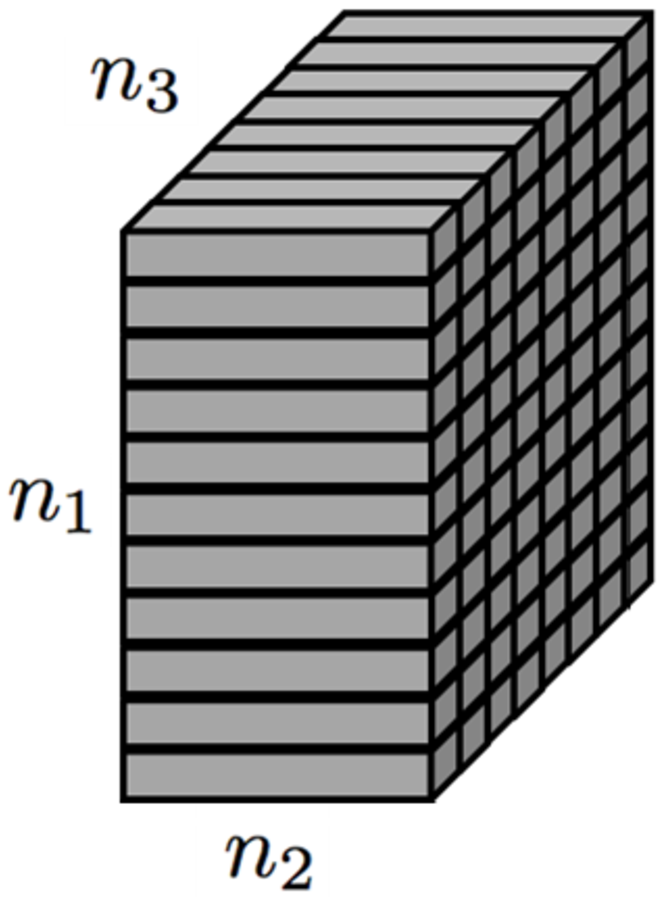}}} &
			
			{\resizebox{\imsize}{!}{
					\includegraphics[width=\imsize]{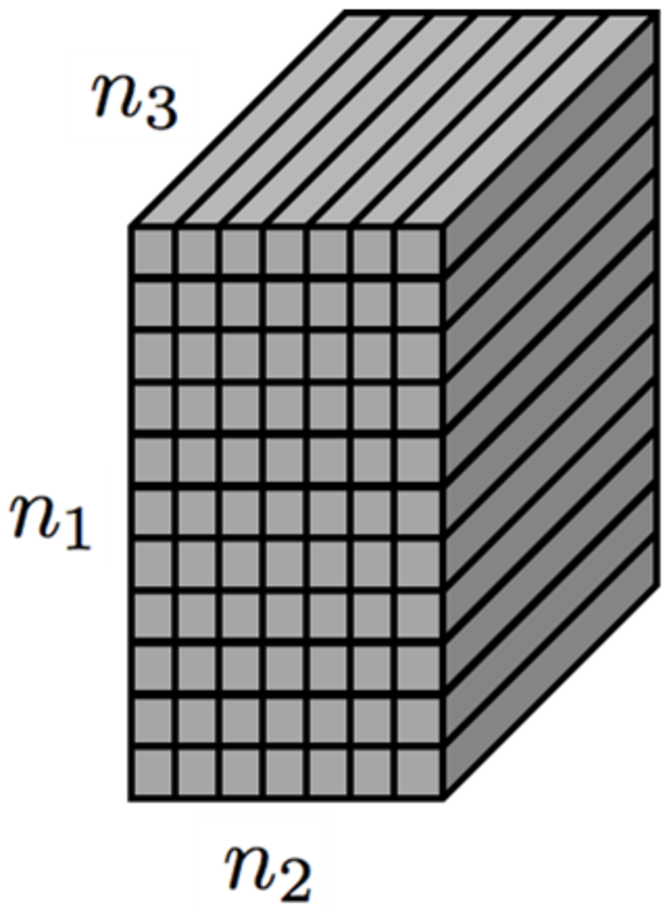}}} &
			
			{\resizebox{\imsize}{!}{
					\includegraphics[width=\imsize]{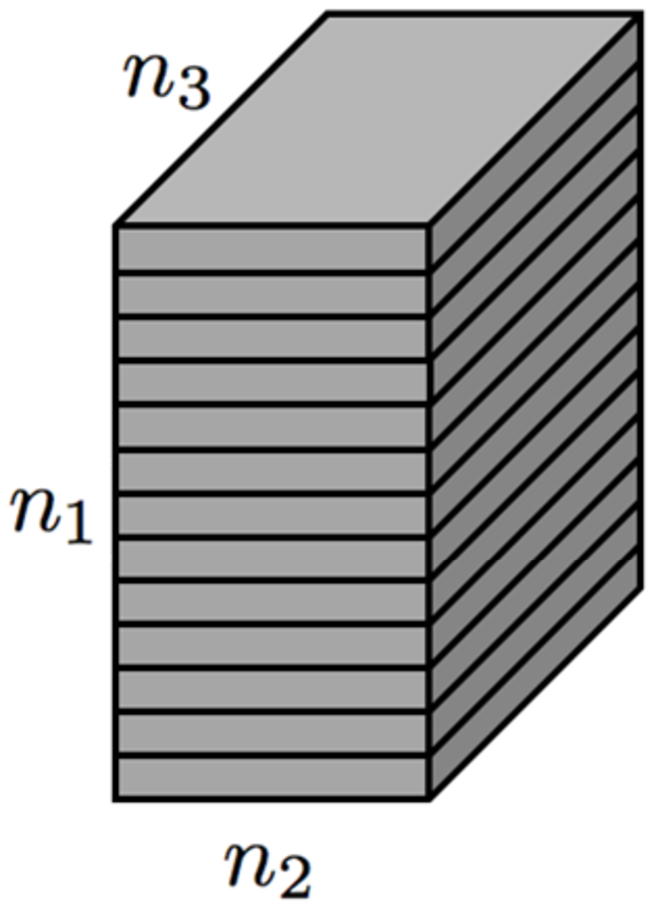}}} &
			
			{\resizebox{\imsize}{!}{
					\includegraphics[width=\imsize]{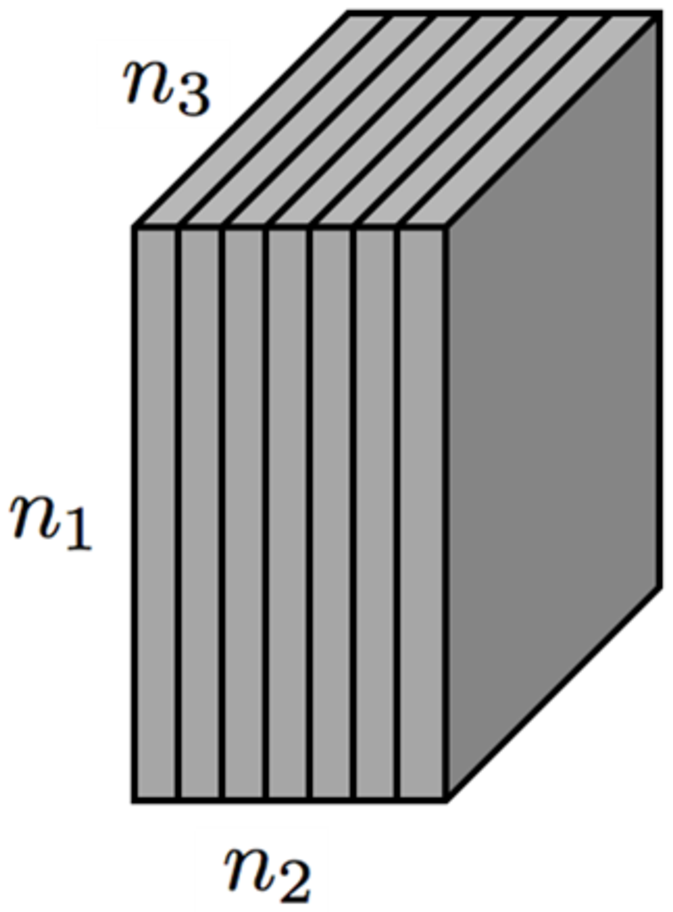}}} &
			
			{\resizebox{\imsize}{!}{
					\includegraphics[width=\imsize]{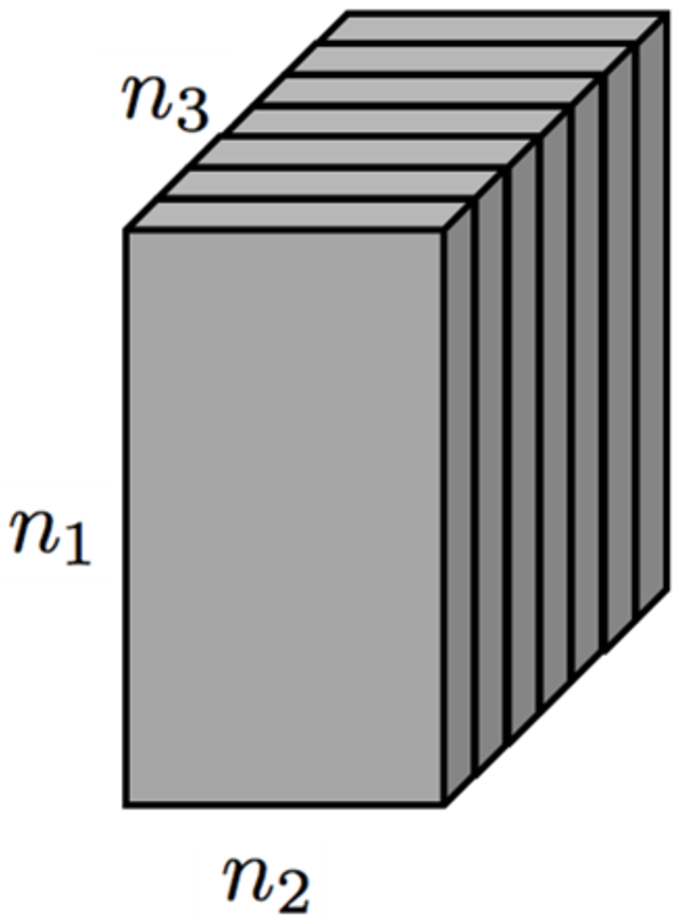}}} \\
			
			\small (a) & \small  (b) & \small  (c) & \small  (d) & \small  (e) & \small  (f)
		\end{tabular}
	\end{center}
	\caption{\small Different views of a third-order tensor $\AA \in \spCtensor$. (a) column fibers: $\AA(:,j,k)$; (b) row fibers: $\AA(i,:,k)$; (c) tube fibers: $\AA(i,j,:)$; (d) horizontal slices: $\AA(i,:,:)$; (e) lateral slices: $\AA(:,j,:)$; (f) frontal slices: $\AA(:,:,k)$ \label{fig:tensors}}
\end{figure}
We direct the reader now to Figure~\ref{fig:tensors} for different ``views" of a third-order tensor, which will be useful in visualizing the forthcoming concepts. We also make use of some notions from block matrices.  Define the standard block unit vectors $\vE_k^{np \times n} := \vehat_k^p \otimes I_{n \times n}$, where $\vehat_k^p \in \spC^p$ is the vector of all zeros except for the $k$th entry, and $I_{n \times n}$ is the identity in $\spCnn$.  When the dimensions are clear from context, we drop the superscripts.  See \eqref{eq:E1} for various ways of expressing $\vE_1^{np \times n}$.
\begin{equation} \label{eq:E1}
\vE_1^{np \times n}
=
\begin{bmatrix}
I_{n \times n} \\
0 \\
\vdots \\
0
\end{bmatrix}
=
\begin{bmatrix}
1 \\
0 \\
\vdots \\
0
\end{bmatrix}
\otimes
I_{n \times n}
=
\unfold{\II_{n \times n \times p}},
\end{equation}
where $\texttt{unfold}$ is defined shortly.

In \cite{Braman2010, KilmerBramanHaoHoover2013, KilmerMartin2011}, a paradigm is proposed for multiplying third-order tensors, based on viewing a tensor as a stack of frontal slices (as in Figure~\ref{fig:tensors}(f)).  We consider a tensor $\AA$ of size $m \times n \times p$ and $\BB$ of size $n \times s \times p$ and denote their frontal faces respectively as $A^{(k)}$ and $B^{(k)}$, $k = 1, \ldots, p$.  We also define the operations $\texttt{bcirc}$, $\texttt{unfold}$, $\texttt{fold}$, as
\begin{equation} \label{eq:bcircA}
\bcirc{\AA} :=
\begin{bmatrix}
A^{(1)}	& A^{(p)}	& A^{(p-1)}		& \cdots	& A^{(2)}	\\
A^{(2)}	& A^{(1)}	& A^{(p)}		& \cdots	& A^{(3)}	\\
\vdots	& \ddots	& \ddots		& \ddots	& \vdots	\\
A^{(p)}	& A^{(p-1)}	& \ddots		& A^{(2)}	& A^{(1)}
\end{bmatrix},
\end{equation}
\[
\unfold{\AA} :=
\begin{bmatrix}
A^{(1)}	\\
A^{(2)}	\\
\vdots	\\
A^{(p)}
\end{bmatrix}, \mbox{ and } \fold{\unfold{\AA}} := \AA.
\]
The {\em t-product} of two tensors $\AA$ and $\BB$ is then given as
\[
\AA * \BB := \fold{\bcirc{\AA} \unfold{\BB}}.
\]
Note that the operators \texttt{fold}, \texttt{unfold}, and \texttt{bcirc} are linear.

The notion of transposition is defined face-wise, i.e., $\AA^*$ is the $n \times m \times p$ tensor obtained by taking the conjugate transpose of each frontal slice of $\AA$ and then reversing the order of the second through $p$th transposed slices.

For tensors with $n \times n$ square faces, there is a tensor identity $\II_{n \times n \times p} \in \spC^{n \times n \times p}$, whose first frontal slice is the $n \times n$ identity matrix and whose remaining frontal slices are all the zero matrix.  With $\II_{n \times n \times p}$, One can then define the notion of an inverse with respect to the t-product.  Namely, $\AA, \BB \in \spC^{n \times n \times p}$ are inverses of each other if $\AA * \BB = \II_{n \times n \times p}$ and $\BB * \AA = \II_{n \times n \times p}$.  The t-product formalism further gives rise to its own notion of polynomials, with powers of tensors defined as $\AA^j := \underbrace{\AA * \cdots * \AA}_{j~\text{times}}$.

Assuming that $\AA \in \spC^{n \times n \times p}$ has diagonalizable faces, we can also define a tensor eigendecomposition.  That is, we have that $A^{(k)} = X^{(k)} D^{(k)} (X^{(k)})^\inv$, for all $k = 1, \ldots, p$, and define $\XX$ and $\DD$ to be the tensors whose faces are $X^{(k)}$ and $D^{(k)}$, respectively.  Then
\begin{equation} \label{def:t-eigendecomp}
\AA = \XX * \DD * \XX^\inv \mbox{ and } \AA * \XXarrow_i = \XXarrow_i * \vd_i,
\end{equation}
where $\XXarrow_i$ are the $n \times 1 \times p$ lateral slices of $\XX$ (see Figure~\ref{fig:tensors}(e) ) and $\vd_j$ are the $1 \times 1 \times p$ tubal fibers of $\DD$ (see Figure~\ref{fig:tensors}).  We say that $\DD$ is {\em f-diagonal}, i.e., that each of its frontal faces is a diagonal matrix.

The eigenvalue decomposition~\eqref{def:t-eigendecomp} is not unique.  See \cite{GleichGreifVarah2013} for an alternative circulant-based interpretation of third-order tensors, as well as a deeper exploration of a unique canonical eigendecomposition for tensors.

\subsection{The tensor t-exponential} \label{subsec:tensor_t-exp}
As motivation, we consider the solution to a multidimensional ordinary differential equation.  Suppose that $\AA$ has square frontal faces, i.e., that $\AA \in \spC^{n \times n \times p}$ and let $\BB: [0, \infty) \to \spC^{n \times s \times p}$ be an unknown function with $\BB(0)$ given.  With $\frac{\d}{\d t}$ acting element-wise, we consider the differential equation
\begin{equation} \label{eq:diffeq}
\frac{\d \BB}{\d t}(t) = \AA * \BB(t).
\end{equation}
Unfolding both sides leads to
\[
\frac{\d}{\d t}
\begin{bmatrix}
B^{(1)}(t) \\
\vdots \\
B^{(n)}(t)
\end{bmatrix}
=
\bcirc{\AA}
\begin{bmatrix}
B^{(1)}(t) \\
\vdots \\
B^{(n)}(t)
\end{bmatrix},
\]
whose solution can be expressed in terms of the matrix exponential as
\[
\begin{bmatrix}
B^{(1)}(t) \\
\vdots \\
B^{(n)}(t)
\end{bmatrix}
=
\exp(\bcirc{\AA} t)
\begin{bmatrix}
B^{(1)}(0) \\
\vdots \\
B^{(n)}(0)
\end{bmatrix}.
\]
Folding both sides again leads to the {\em tensor t-exponential},
\begin{equation} \label{eq:t-exp}
\BB(t) = \fold{\exp(\AA t) \unfold{\BB(0)}} =: \exp(\AA t) * \BB(0).
\end{equation}

\subsection{The tensor t-function} \label{subsec:tensor_t-function}
Using the tensor t-exponential as inspiration, we can define a more general notion for the scalar function $f$ of a tensor $\AA \in \spC^{n \times n \times p}$ multiplied by a tensor $\BB \in \spC^{n \times s \times p}$ as
\begin{equation} \label{def:fAABB}
\fAABB := \fold{f(\bcirc{\AA}) \cdot \unfold{\BB}},
\end{equation}
which we call the {\em tensor t-function}. Note that $f(\bcirc{\AA}) \cdot \unfold{\BB}$ is merely a matrix function times a block vector. If $\BB = \II_{n \times n \times p}$, then by equation~\eqref{eq:E1^npxn} the definition for $f(\AA)$ reduces to
\begin{equation} \label{def:fAA}
f(\AA) := \fold{f(\bcirc{\AA}) \vE_1^{np \times n}}.
\end{equation}
A natural question is whether the definition~\eqref{def:fAABB} behaves ``as expected" in common scenarios.  To answer this question, we require some results on block circulant matrices and the tensor t-product.
\begin{theorem}[Theorem 5.6.5 in \cite{Davis2012}] \label{thm:block_circ_space}
	Suppose $A, B \in \spC^{np \times np}$ are block circulant matrices with $n \times n$ blocks.  Let $\{\alpha_j\}_{j=1}^k$ be scalars.  Then $A^T$, $A^*$, $\alpha_1 A + \alpha_2 B$, $AB$, $q(A) = \sum_{j=1}^{k} \alpha_j A^j$, and $A^\inv$ (when it exists) are also block circulant.
\end{theorem}
\begin{remark} \label{rem:bcirc_E1_rep}
	From \eqref{eq:bcircA}, we can see that any block circulant matrix $C \in \spC^{np \times np}$ can be represented by its first column $C \vE_1^{np \times n}$.  Let $\CC \in \spC^{n \times n \times p}$ be a tensor whose frontal faces are the block entries of $C \vE_1^{np \times n}$.  Then $\CC = \fold{C \vE_1^{np \times n}}$.
\end{remark}
\begin{lemma} \label{lem:bcirc_props}
	Let $\AA \in \spC^{m \times n \times p}$ and $\BB \in \spC^{n \times s \times p}$.  Then
	\begin{enumerate}[(i)]
		\item $\unfold{\AA} = \bcirc{\AA} \vE_1^{np \times n}$; \label{bcirc:trivial}
		\item $\bcirc{\fold{\bcirc{\AA} \vE_1^{np \times n}}} = \bcirc{\AA}$; \label{bcirc:fold}
		\item $\bcirc{\AA * \BB} = \bcirc{\AA} \bcirc{\BB}$; \label{bcirc:products}
		\item $\bcirc{\AA}^j = \bcirc{\AA^j}$, for all $j = 0, 1, \ldots$; and \label{bcirc:powers}
		\item $(\AA * \BB)^* = \BB^* * \AA^*$. \label{bcirc:transpose}
	\end{enumerate}
\end{lemma}
\begin{proof}
	We drop the superscripts on $\vE_1^{np \times n}$ for ease of presentation. Parts~\eqref{bcirc:trivial} and \eqref{bcirc:fold} follow from Remark~\eqref{rem:bcirc_E1_rep}. To prove part~\eqref{bcirc:products}, we note by part~\eqref{bcirc:trivial} that
	\begin{align*}
	\bcirc{\AA * \BB}
	&= \bcirc{\fold{\bcirc{\AA} \unfold{\BB}}} \\
	&= \bcirc{\fold{\bcirc{\AA} \bcirc{\BB} \vE_1}}.
	\end{align*}
	Note that $\bcirc{\AA} \bcirc{\BB}$ is a block circulant matrix by Theorem~\ref{thm:block_circ_space}. Then by part~\eqref{bcirc:fold},
	\[
	\bcirc{\fold{\bcirc{\AA} \bcirc{\BB} \vE_1}} = \bcirc{\AA} \bcirc{\BB}.
	\]
	Part~\eqref{bcirc:powers} follows by induction on part~\eqref{bcirc:products}.  Part~\eqref{bcirc:transpose} is the same as~\cite[Lemma 3.16]{KilmerMartin2011}.
\end{proof}

Let $\DD$ be an $n \times n \times p$ {\em f-diagonal} tensor, i.e., a tensor whose $n \times n$ frontal slices are diagonal matrices.  Alternatively, one can think of such a tensor as an $n \times n$ matrix nonzero tube fibers on the diagonal, and zero tube fibers everywhere else.  (Reference Figure~\ref{fig:tensors}(c).)  The following theorem summarizes the relationship between the block circulant of $\DD$ and those of its tube fibers.
\begin{theorem} \label{thm:spec_fdiagonal_tubefibers}
	Let $\DD \in \spC^{n \times n \times p}$ be f-diagonal, and let $\{\vd_i\}_{i=1}^n \subset \spC^{1 \times 1 \times p}$ denote its diagonal tube fibers.  Then the spectrum of $\bcirc{\DD}$ is identical to the union of the spectra of $\bcirc{\vd_i}$, $i = 1, \ldots n$.
\end{theorem}
\begin{proof}
\newcommand{\Dface}[1]{\begin{matrix} d_{1}^{(#1)} & & \\ & \ddots & \\ & & d_{n}^{(#1)} \end{matrix}}
\newcommand{\Dfacehl}[2]{\begin{matrix} \mathcolorbox{#1}{d_{1}^{(#2)}} & & \\ & \ddots & \\ & & d_{n}^{(#2)} \end{matrix}}
\newcommand{\Dfacehlfirst}[2]{\begin{matrix} \mathcolorbox{#1}{d_{1}^{(#2)}} & & \\ & \ddots & \\ & & 0 \end{matrix}}
\newcommand{\Delementhl}[2]{\mathcolorbox{#1}{d_{1}^{(#2)}}}
	We begin by deriving an expression for $\bcirc{\DD}$ in terms of the $p \times p$ circulant matrices $\bcirc{\vd_i}$.  Denote each slice as $D^{(k)}$, $k = 1, \ldots, p$, with diagonal entries denoted as $d_{i}^{(k)}$, for $i = 1, \ldots, n$; i.e.,
	\[
	D^{(k)}
	= \begin{bmatrix} d_{1}^{(k)} & & \\ & \ddots & \\ & & d_{n}^{(k)} \end{bmatrix}.
	\]
	
	Then we can express $\bcirc{\DD}$ as follows:
	\begin{align*}
	&\bcirc{\DD} \\
	&= \begin{bmatrix}
		D^{(1)}	& D^{(p)} & \cdots  & D^{(2)} \\
		D^{(2)} & D^{(1)} & \ddots  & \vdots \\
		\vdots  & \ddots  & \ddots  & D^{(p)} \\
		D^{(p)} & \cdots & D^{(2)} & D^{(1)} \\
	\end{bmatrix} \\
	&=
	\begin{bmatrix}
		\Dfacehl{SunsetOrangeLL}{1} &
		\Dfacehl{SunsetPurpleLL}{p} &
		\cdots &
		\Dfacehl{SunsetRedLL}{2} \\
		\Dfacehl{SunsetRedLL}{2} &
		\Dfacehl{SunsetOrangeLL}{1} &
		\ddots &
		\vdots \\
		\vdots &
		\ddots &
		\ddots &
		\Dfacehl{SunsetPurpleLL}{p} \\
		\Dfacehl{SunsetPurpleLL}{p} &
		\cdots &
		\Dfacehl{SunsetRedLL}{2} &
		\Dfacehl{SunsetOrangeLL}{1}
	\end{bmatrix}.
	\end{align*}
	Collecting the highlighted elements, note that the block circulant of the first tube fiber is given as
	\[
	\bcirc{\vd_1}
	=
	\begin{bmatrix}
	\Delementhl{SunsetOrangeLL}{1} &
	\Delementhl{SunsetPurpleLL}{p} &
	\cdots &
	\Delementhl{SunsetRedLL}{2} \\
	\Delementhl{SunsetRedLL}{2} &
	\Delementhl{SunsetOrangeLL}{1} &
	\ddots &
	\vdots \\
	\vdots &
	\ddots &
	\ddots &
	\Delementhl{SunsetPurpleLL}{p} \\
	\Delementhl{SunsetPurpleLL}{p} &
	\cdots &
	\Delementhl{SunsetRedLL}{2} &
	\Delementhl{SunsetOrangeLL}{1}
	\end{bmatrix}.
	\]
	Defining
	\newcommand{\Ihat}{\widehat{I}}
	\[
	\Ihat_1
	:=
	\begin{bmatrix}
	1	&	&		 &	\\
		& 0	&		 &	\\
		&	& \ddots &	\\
		&	&		 & 0
	\end{bmatrix}
	\in \spCnn,
	\]
	it holds that
	\begin{align*}
	& \bcirc{\vd_1} \otimes \Ihat_1 \\
	& =
	\begin{bmatrix}
	\Dfacehlfirst{SunsetOrangeLL}{1} &
	\Dfacehlfirst{SunsetPurpleLL}{p} &
	\cdots &
	\Dfacehlfirst{SunsetRedLL}{2} \\
	\Dfacehlfirst{SunsetRedLL}{2} &
	\Dfacehlfirst{SunsetOrangeLL}{1} &
	\ddots &
	\vdots \\
	\vdots &
	\ddots &
	\ddots &
	\Dfacehlfirst{SunsetPurpleLL}{p} \\
	\Dfacehlfirst{SunsetPurpleLL}{p} &
	\cdots &
	\Dfacehlfirst{SunsetRedLL}{2} &
	\Dfacehlfirst{SunsetOrangeLL}{1}
	\end{bmatrix}
	\end{align*}
	Noting the same pattern for each $i = 1, \ldots, n$, it is not hard to see that
	\begin{equation}
	\bcirc{\DD} = \sum_{i=1}^{n} \bcirc{\vd_i} \otimes \Ihat_i,
	\end{equation}
	where $\Ihat_i \in \spCnn$ is zero everywhere except for the $ii$th entry, which is one.
	
	It is known that a circulant matrix is unitarily diagonalizable by the discrete Fourier transform (DFT); see, e.g., \cite[Section 3.2]{Davis2012}.  Then for a $p \times p$ circulant matrix $C$, and with $F_p$ denoting the $p \times p$ DFT, $F_p^* C F_p = \Lambda$, where $\Lambda \in \spC^{p \times p}$ is diagonal.  Since each $\bcirc{\vd_i}$ is a $p \times p$ circulant matrix, there exists for each $i = 1, \ldots, n$ a diagonal $\Lambda_i \in \spC^{p \times p}$ such that
	\begin{equation} \label{eq:diagonalize_bcirc_tubes}
	F_p \bcirc{\vd_i} F_p = \Lambda_i.
	\end{equation}
	
	We also have the following useful property of the Kronecker product for matrices $A$, $B$, $C$, and $D$ such that the products $AC$ and $BD$ exist; see, e.g., \cite[Lemma 4.2.10]{HornJohnson1991}:
	\begin{equation} \label{eq:Kronecker_product}
	(A \otimes B) (C \otimes D) = (AB) \times (CD)
	\end{equation}
	
	Consequently,
	\begin{align*}
	(F_p \otimes I_{n \times n}) \bcirc{\DD} (F_p^* \otimes I_{n \times n})
	&= (F_p \otimes I_{n \times n}) \left(\sum_{i=1}^n \bcirc{\vd_i} \otimes \Ihat_i\right) (F_p^* \otimes I_{n \times n}) \\
	&= \sum_{i=1}^n (F_p \otimes I_{n \times n}) \bcirc{\vd_i} \otimes \Ihat_i (F_p^* \otimes I_{n \times n}) \\
	&= \sum_{i=1}^n (F_p \bcirc{\vd_i} F_p^*) \otimes (I_{n \times n} \Ihat_i I_{n \times n}), \mbox{ by \eqref{eq:Kronecker_product}} \\
	&= \sum_{i=1}^n \Lambda_i \otimes \Ihat_i, \mbox{ by \eqref{eq:diagonalize_bcirc_tubes}}.
	\end{align*}
	Noting that $F_p \otimes I_{n \times n}$ is unitary and that the matrix $\Lambda := \sum_{i=1}^n \Lambda_i \otimes \Ihat_i$ is a diagonal matrix whose entries are precisely the diagonal entries of all the $\Lambda_i$ concludes the proof.
\end{proof}

\begin{corollary} \label{cor:spec_fdiagonal_tubefibers}
	Let $\DD \in \spC^{n \times n \times p}$ be f-diagonal, and let $\{\vd_i\}_{i=1}^n \subset \spC^{1 \times 1 \times p}$ denote its diagonal tube fibers.  Then a function $f$ being defined on the spectrum of $\bcirc{\DD}$ is equivalent to $f$ being defined on the union of the spectra of $\bcirc{\vd_i}$, $i = 1, \ldots, n$.
\end{corollary}

An immediate consequence of Theorem~\ref{thm:spec_fdiagonal_tubefibers} is that a function $f$ being defined on the spectrum of $\bcirc{\DD}$ is equivalent to $f$ being defined on the spectra of $\bcirc{\vd_i}$, $i = 1, \ldots, n$.  The interpolating polynomials for $f(\bcirc{\DD})$ and $f(\bcirc{\vd_i})$, $i = 1, \ldots, n$, are also related.

The following theorem ensures that definition~\eqref{def:fAABB} is well defined when $f$ is a polynomial, when $\AA$ and $\BB$ are second-order tensors (i.e., matrices), and when $f$ is the inverse function.
\begin{theorem} \label{thm:tensor_t-func_consistency}
	Let $\AA \in \spC^{n \times n \times p}$ and $\BB \in \spC^{n \times s \times p}$.
	\begin{enumerate}[(i)]
		\item If $f \equiv q$, where $q$ is a polynomial, then the tensor t-function definition \eqref{def:fAABB} matches the polynomial notion in the t-product formalism, i.e., \label{consistency:poly}
		\[
		\fold{q(\bcirc{\AA}) \cdot \unfold{\BB}} = \fold{\bcirc{q(\AA)} \cdot \unfold{\BB}}.
		\] 
		\item Let $q$ be the scalar polynomial guaranteed by Definition~\ref{def:mat_func_Hermite} so that \linebreak$f(\bcirc{\AA}) = q(\bcirc{\AA})$.  Then $\fAABB = q(\AA) * \BB$. \label{consistency:Hermite}
		\item If $\AA$ is a matrix and $\BB$ a block vector (i.e., if $p = 1$), then $\fAABB$ reduces to the usual matrix function definition. \label{consistency:mat_func}
		\item If $f(z) = z^\inv$, then $f(\AA) * \AA = \AA * f(\AA) = \II_{n \times n \times p}$. \label{consistency:inverse}
	\end{enumerate}
\end{theorem}
\begin{proof}
	For part~\eqref{consistency:poly}, let $q(z) = \sum_{j=1}^{m} c_j z^j$.  Then by Lemma~\ref{lem:bcirc_props}\eqref{bcirc:powers} and the linearity of $\texttt{fold}$, we have that
	\vspace{-8pt}
	\begin{align*}
	\fold{q(\bcirc{\AA}) \cdot \unfold{\BB}}
	&= \fold{\sum_{j=1}^{m} c_j \bcirc{\AA}^j \cdot \unfold{\BB}} \\
	&= \sum_{j=1}^{m} c_j \fold{\bcirc{\AA^j} \cdot \unfold{\BB}} \\
	&= \sum_{j=1}^{m} c_j \bcirc{\AA^j} * \BB \\
	&= \fold{\bcirc{q(\AA)} \cdot \unfold{\BB}}.
	\end{align*}
	Part~\eqref{consistency:Hermite} is a special case of part~\eqref{consistency:poly}.
	As for part~\eqref{consistency:mat_func}, since $p=1$, we have that $\fold{\AA} = \bcirc{\AA} = \AA = \unfold{\AA}$, and similarly for $\BB$.  Then the definition of $\fAABB$ reduces immediately to the matrix function case.
	Part~\eqref{consistency:inverse} follows by carefully unwrapping the definition of $f(\AA)$:
	\begin{align*}
	f(\AA) * \AA
	&= \fold{\bcirc{\AA}^\inv \unfold{\AA}} \\
	&= \fold{\bcirc{\AA}^\inv \bcirc{\AA} \vE_1^{np \times n}}, \mbox{by Lemma~\ref{lem:bcirc_props}\eqref{bcirc:trivial}} \\
	&= \fold{\vE_1^{np \times n}} = \II_{n \times n \times p}.
	\end{align*}
	Likewise with the other product:
	\begin{align*}
	\AA * f(\AA)
	&= \fold{\bcirc{\AA} \unfold{\fold{\bcirc{\AA}^\inv \unfold{\II_{n \times n \times p}}}}} \\
	&= \fold{\bcirc{\AA} \bcirc{\AA}^\inv \vE_1^{np \times n}} \\
	&= \fold{\vE_1^{np \times n}} = \II_{n \times n \times p}.
	\end{align*}
\end{proof}

The definition~\eqref{def:fAABB} possesses generalized versions of many of the core properties of matrix functions.
\begin{theorem} \label{thm:tensor_t-func_props}
	Let $\AA \in \spC^{n \times n \times p}$, and let $f: \spC \to \spC$ be defined on a region in the complex plane containing the spectrum of $\bcirc{\AA}$.  For part~(iv), assume that $\AA$ has an eigendecomposition as in equation~\eqref{def:t-eigendecomp}, with $\AA * \XXarrow_i = \DD * \XXarrow_i = \XXarrow_i * \vd_i$, $i = 1, \ldots, n$. Then it holds that
	\begin{enumerate}[(i)]
		\item $f(\AA)$ commutes with $\AA$; \label{tfunc:commutativity}
		\item $f(\AA^*) = f(\AA)^*$; \label{tfunc:transpose}
		\item $f(\XX * \AA * \XX^\inv) = \XX f(\AA) \XX^\inv$; and \label{tfunc:similarity}
		\item $f(\DD) * \XXarrow_i = \XXarrow_i * f(\vd_i)$, for all $i = 1, \ldots, n$. \label{tfunc:diagonal}
	\end{enumerate}
\end{theorem}
\begin{proof}
	For parts~\eqref{tfunc:commutativity}-\eqref{tfunc:similarity}, it suffices by Theorem~\ref{thm:tensor_t-func_consistency}\eqref{consistency:Hermite} to show that the statements hold for $f(z) = \sum_{j=1}^{m} c_j z^j$.
	Part~\eqref{tfunc:commutativity} then follows immediately.
	To prove part~\eqref{tfunc:transpose}, we need only show that $(\AA^j)^* = (\AA^*)^j$ for all $j = 0, 1, \ldots$, which follows by induction from  Lemma~\ref{lem:bcirc_props}\eqref{bcirc:transpose}.
	Part~\eqref{tfunc:similarity} also follows inductively.  The base cases $j = 0, 1$ clearly hold.  Assume for some $j = k$, $(\XX * \AA * \XX^\inv)^k = \XX (\AA)^k \XX^\inv$, and then note that
	\begin{align*}
	(\XX * \AA * \XX^\inv)^{k+1}
	&= (\XX * \AA * \XX^\inv)^k * (\XX * \AA * \XX^\inv)\\
	&= \XX * (\AA)^k * \XX^\inv * \XX * \AA * \XX^\inv = \XX * (\AA)^{k+1} * \XX^\inv.
	\end{align*}
	
	For part~\eqref{tfunc:diagonal}, we fix $i \in \{i, \ldots, n \}$. By Corollary~\ref{cor:spec_fdiagonal_tubefibers}, $f$ being defined on $\spec(\bcirc{\DD})$ implies that it is also defined on $\spec(\bcirc{\vd_i})$.  Let $q$ and $q_i$ be the polynomials guaranteed by Theorem~\ref{def:mat_func_Hermite} such that $f(\bcirc{\DD}) = q(\bcirc{\DD})$ and $f(\bcirc{\vd_i}) = q_i(\bcirc{\vd_i})$.  By Theorem~\ref{thm:spec_fdiagonal_tubefibers}, $\spec(\bcirc{\vd_i}) \subset \spec(\bcirc{\DD})$, so by Theorem~\ref{thm:spec_fdiagonal_tubefibers}, it follows that $q_i(\bcirc{\vd_i}) = q(\bcirc{\vd_i})$.  Then it suffices to prove part~\eqref{tfunc:diagonal} for $\DD^j$, $j = 1, 0, \ldots$.  The cases $j = 0, 1$ clearly hold, and we assume the statement holds for some $j = k \geq 1$.  Then
	\[
	\DD^{k+1} * \XXarrow_i = \DD * (\DD^k * \XXarrow_i) = \DD * \XXarrow_i * \vd_i^k = \XXarrow_i * \vd_i^{k+1}.
	\]
\end{proof}

\begin{remark} \label{rem:tfunc_eigendecomp}
	{\rm
		When $\AA$ has an eigendecomposition $\XX * \DD * \XX^\inv$ as in \eqref{def:t-eigendecomp}, then by Theorems~\ref{thm:spec_fdiagonal_tubefibers} and \ref{thm:tensor_t-func_props}, an equivalent definition for $f(\AA)$ is given as
		\[
		f(\AA) = \XX *
		\begin{bmatrix}
		f(\vd_1)	&			&		\\
		& \ddots	&		\\
		&			& f(\vd_n)
		\end{bmatrix}
		* \XX^\inv,
		\]
		where the inner matrix should be regarded three-dimensionally, with its elements being tube fibers (\cf Figure~\ref{fig:tensors}(c)).
	}
\end{remark}

\section{Centrality and communicability of a third-order network} \label{sec:network_analysis}
We use the term {\em network} to denote an undirected, unweighted graph with $n$ nodes.  The graph, and by extension, the network, can be represented by its {\em adjacency matrix} $A \in \spR^{n \times n}$.  The $ij$th entry of $A$ is 1 if nodes $i$ and $j$ are connected, and 0 otherwise.  As a rule, a node is not connected to itself, so $A_{ii} = 0$.  The centrality of the $i$th node is defined as $\exp(A)_{ii}$, while the communicability between nodes $i$ and $j$ is defined as $\exp(A)_{ij}$.

These notions can be extended to higher-order situations.  Suppose we are concerned instead about triplets, instead of pairs, of nodes.  Then it is possible to construct an adjacency tensor $\AA$, where a 1 at entry $\AA_{ijk}$ indicates that distinct nodes $i$, $j$, and $k$ are connected and 0 otherwise.  Information will, however, be lost if only the adjacency tensor is considered, since pairwise connectivity is stored in the adjacency matrix.  Multilayer networks, such as those describing a city's bus, metro, and tram systems, constitute a more natural application, and several notions of centrality are explored in \cite{DeDomenicoSoleRibaltaOmodei2013}.  Alternatively, it is not hard to imagine a time-dependent network stored as a tensor, where each frontal face corresponds to a sampling of the network at discrete times; see, e.g., \cite{ZhangElyAeron2014}.  In any of these situations, we could compute the communicability of a triple as $\exp(\AA)_{ijk}$, where $\exp(\AA)$ is our tensor t-exponential.  Centrality for a node $i$ would thus be defined as $\exp(\AA)_{iii}$.

\section{Computing the tensor t-function} \label{sec:methods_compute}
For the tensor t-function to be viable in practice, we need efficient methods for approximating $\fAABB$ numerically.  The t-eigendecomposition and t-Krylov methods of \cite{KilmerBramanHaoHoover2013} are potential options, but a full eigendecomposition may be expensive to compute for large tensors, and crafting t-Krylov methods for tensor functions may not be necessary, given the equivalence between $\fAABB$ and the $\fAB$ problem.  We therefore adapt the block Krylov subspace methods (KSMs) of \cite{FrommerLundSzyld2017}, designed specifically for such problems, and further optimize the computations by taking advantage of the block circulant structure of $\bcirc{\AA}$.

\subsection{Block Krylov subspace methods for matrix functions} \label{subsec:block_Krylov_framework}
We recount here the block Krylov subspace framework from \cite{FrommerLundSzyld2017}, given a scalar function $f$ defined on an $np \times np$ matrix $A$ and a block vector $\vB \in \spCNn$.  This framework allows us to treat different inner products and norms simultaneously, some of which are computationally more advantageous than others; see Section~\ref{subsec:complexity}.

Let $\spS \subset \spCnn$ be a *-subalgebra with identity.
\begin{definition} \label{def:block_IPS}
	A mapping $\IPSnoarg$ from $\spCNn \times \spCNn$ to $\spS$ is called a {\em block inner product onto $\spS$} if it satisfies the following conditions for all $\vX,\vY,\vZ \in \spCNn$ and $C \in \spS$:
	\begin{enumerate}[(i)]
		\item \textit{$\spS$-linearity}: $\IPS{\vX+\vY}{\vZ C} = \IPS{\vX}{\vZ} C + \IPS{\vY}{\vZ} C$; \label{S:linearity}
		\item \textit{symmetry}: $\IPS{\vX}{\vY} = \IPS{\vY}{\vX}^*$; \label{S:symmetry}
		\item \textit{definiteness}: $\IPS{\vX}{\vX}$ is positive definite if $\vX$ has full rank, and $\IPS{\vX}{\vX} = 0$ if and only if $\vX = 0$. \label{S:definiteness}
	\end{enumerate}
\end{definition}

\begin{definition} \label{def:scal_quot}
	A mapping $\N$ which maps all $\vX \in \spCNn$ with full rank on a matrix $\N(\vX) \in \spS$ is called a {\em scaling quotient} if for all such $\vX$, there exists $\vY \in \spCNn$ such that $\vX = \vY \N(\vX)$ and $\IPS{\vY}{\vY} = I_s$. 
\end{definition}

\begin{definition} \label{def:block_ortho_norm} Let $\vX, \vY \in \spCNn$.
	\begin{enumerate}[(i)]
		\item $\vX, \vY$ are {\em $\IPSnoarg$-orthogonal}, if $\IPS{\vX}{\vY} = 0$.
		\item $\vX$ is {\em $\IPSnoarg$-normalized} if $\N(\vX) = I$.
		\item $\{\vX_1, \dots, \vX_m \} \subset \spCNn$ is {\em $\IPSnoarg$-orthonormal} if $\IPS{\vX_i}{\vX_j} = \delta_{ij} I$, where $\delta_{ij}$ is the Kronecker delta.
	\end{enumerate}
\end{definition}

We say that a set of vectors $\{\vX_j\}_{j=1}^m \subset \spCNn$ {\em $\spS$-spans} a space $\spK \subset \spCNn$ and write $\spK = \spnS\{\vX_j\}_{j=1}^m$, where
\[
\spnS\{\vX_j\}_{j=1}^m := \left\{\sum_{j=1}^{m} \vX_j \Gamma_j : \Gamma_j \in \spS \mbox{ for all } j = 1, \ldots, m \right\} \subset \spCNn.
\]
The set $\{\vX_j\}_{j=1}^m$ constitutes an {\em $\IPSnoarg$-orthonormal basis} for $\spK$ if $m$ is the dimension of $\spK$, $\spK = \spnS\{\vX_j\}_{j=1}^m$, and $\{\vX_j\}_{j=1}^m$ are orthonormal.

We define the {\em $m$th block Krylov subspace for $A \in \spCNN$ and $\vB \in \spCNn$} as
\[
\spKS_m(A, \vB) = \spnS\{\vB, A\vB, \ldots, A^{m-1}\vB\}.
\]
There exist many choices for $\spS$, $\IPSnoarg$, and $N$; see \cite{ElbouyahyaouiMessaoudiSadok2008, FrommerLundSzyld2017} for further details.  We consider only the {\em classical} and {\em global} paradigms, because of the potential to speed up convergence and the low computational effort per iteration, respectively:
\begin{center}
	\begin{tabular}{c| c c}
							& classical 					& global			 	\\ \hline
		$\spS$				& $I_n \otimes \spC$			& $\spC I_n$			\\
		$\IPS{\vX}{\vY}$	& $\diag(\vX^*\vY)$				& $\frac{1}{n} \tr{\vX^*\vY} I_n$	\\					
		$\N(\vX)$			& $R$, where $\vX = \vQ R$		& $\frac{1}{\sqrt{n}} \normF{\vX} I_n$
	\end{tabular}
\end{center}
\vspace{6pt}
Algorithm~\ref{alg:block_arnoldi} is the generalized block Arnoldi procedure. We assume that Algorithm~\ref{alg:block_arnoldi} runs to completion without breaking down, i.e., that we obtain
\begin{enumerate}[(i)]
	\item a $\IPSnoarg$-orthonormal basis $\{\vV_k\}_{k=1}^{m+1} \subset \spCNn$, such that each $\vV_k$ has full rank and $\spKS_m(A,\vB) = \spnS\{\vV_k\}_{k=1}^{m}$, and
	\item a block upper Hessenberg matrix $\HH_m \in \spSmm$ and $H_{m+1,m} \in \spS$,
\end{enumerate}
all satisfying the \textit{block Arnoldi relation}
\begin{equation} \label{eq:BAR}
A \bVV_m = \bVV_m \HH_m + \vV_{m+1} H_{m+1,m} \vE_m^*,
\end{equation}
where $\bVV_m = [\vV_1 | \dots | \vV_m] \in \spCNnm$, and $(\HH_m)_{ij} = H_{ij}$. Note that $\HH_m$ has dimension $mn \times mn$; so long as $m \ll p$, $\HH_m$ will be significantly smaller than $A$.  Otherwise, it will be necessary to partition the right-hand side $\vB$ and compute the action of $f(A)$ on each partition separately.  Furthermore, in the global paradigm, $\HH_m$ has a Kronecker structure $H \otimes I_n$, where $H \in \spC^{m \times m}$, so the storage of $\HH_m$ can be reduced.

\begin{algorithm}[htbp!]
	\caption{Block Arnoldi \label{alg:block_arnoldi}}
	\begin{algorithmic}[1]
		\STATE \textbf{Given:} $A$, $\vB$, $\spS$, $\IPSnoarg$, $\N$, $m$
		\STATE Compute $\Beta = \N(\vB)$ and $\vV_1 = \vB B^\inv$ \label{line:norm1}
		\FOR{$k = 1, \dots, m$}
		\STATE Compute $\vW = A \vV_k$
		\FOR{$j = 1, \dots, k$}
		\STATE $H_{j,k} = \IPS{\vV_j}{\vW}$
		\STATE $\vW = \vW - \vV_j H_{j,k}$
		\ENDFOR
		\STATE Compute $H_{k+1,k} = \N(\vW)$ and $\vV_{k+1} = \vW H_{k+1,k}^\inv$ \label{line:norm2}
		\ENDFOR
		\RETURN $\Beta$, $\bVV_m = [\vV_1 | \dots | \vV_m]$, $\HH_m = (H_{j,k})_{j,k = 1}^m$, $\vV_{m+1}$, and $H_{m+1,m}$
	\end{algorithmic}
\end{algorithm}

The paper \cite{FrommerLundSzyld2017} also establishes theory for a block full orthogonalization method for functions of matrices (\bfomfom).  The \bfomfom approximation is defined as
\begin{equation} \label{eq:Fm}
	\vF_m := \bVV_m f(\HH_m) \vE_1 \Beta,
\end{equation}
which indeed reduces to a block FOM approximation when $f(z) = z^\inv$; see, e.g., \cite{Saad2003}.

With the end application being tensors, restarts will be necessary to mitigate memory limitations imposed by handling higher-order data. Restarts for \bfomfom are developed in detail in \cite{FrommerLundSzyld2017} for functions with Cauchy-Stieltjes representations, including the matrix exponential; we present here a high-level summary of the procedure as Algorithm~\ref{alg:bfomfom_restarts}.  Restarts are performed by approximating an error function via adaptive quadrature rules; this step is represented by $\errfunc_m^\k$ in line~\ref{line:errfunc}.
\begin{algorithm}[htbp!]
	\caption{\bfomfom\m: block full orthogonalization method for functions of matrices with restarts \label{alg:bfomfom_restarts}}
	\begin{algorithmic}[1]
		\STATE Given $f$, $A$, $\vB$, $\spS$, $\IPSnoarg$, $\N$, $m$, $t$, \tol
		\STATE Run Algorithm~\ref{alg:block_arnoldi} with inputs $A$, $\vB$, $\spS$, $\IPSnoarg$, $\N$, and $m$ and store \linebreak$\bVV_{m+1}^\one$, $\HHunder_m^\one$, and $\Beta^\one$
		\STATE Compute and store $\vF_m^\one = \bVV_m^\one f\big(\HH_m^\one\big) \vE_1 \Beta$
		\FOR{$k = 1, 2, \ldots$, until convergence}
		\STATE Run Algorithm~\ref{alg:block_arnoldi} with inputs $A$, $\vV_{m+1}^\k$, $\spS$, $\IPSnoarg$, $\N$, and $m$ and \linebreak store $\bVV_{m+1}^\kmore$ in place of the previous basis
		\STATE Compute error approximation $\errapprox_m^\k$ \label{line:errfunc}
		\STATE Update $\vF_m^\kmore := \vF_m^\k + \errapprox_m^\k$
		\ENDFOR
		\RETURN $\vF_m^\kmore$
	\end{algorithmic}
\end{algorithm}

\subsection{Computational complexity and optimizations} \label{subsec:complexity}
Operation counts for \bfomfom are not provided in \cite{FrommerLundSzyld2017}, so we present our own calculations here.  We first consider the two versions of block Arnoldi, classical and global, and then look at single cycle of \bfomfom.

\paragraph{Sparse matrix times a block vector}
We assume $A$ is block circulant, where each $n \times n$ block is sparse with $\bigO(n)$ nonzero entries.  We also assume that $p = \bigO(n)$.  Then $A$ itself can be stored in $\bigO(n^2)$ entries and does not need to be formed explicitly; see also the discussion in \cite{KilmerBramanHaoHoover2013}.  However, the action of $A$ on a block vector will cost the same as for a matrix with $\bigO(n^3)$ nonzero entries, because the full block circulant has $p^2$ blocks, so the product $A \vV$ has complexity $\bigO(n^4)$.

\paragraph{Classical Block Arnoldi}
In addition to the action of $A$ on block vectors, the main kernels of the classical version of Algorithm~\ref{alg:block_arnoldi} are the QR factorization for block vectors, the inner product $\vX^*\vY$, and multiplication between block vectors and small square matrices of the form $\vV C$.  The Householder QR factorization costs $\bigO(n^6)$ \cite{TrefethenBau1997}, the inner product $\bigO(n^4)$, and the $\vV C$ product  $\bigO(n^4)$.  Running $m$ steps of Algorithm~\ref{alg:block_arnoldi} requires $m+1$ QR factorizations, $m$ $A \vV$ products, and $\frac{1}{2} m (m+1)$ inner products and $\vV C$ products, with some negligible addition, for a total of
\begin{equation} \label{O:cl_arnoldi}
	\cost_{\text{cl-Arnoldi}} = \bigO\left((m+1)n^6 + m (m+2)n^4 \right).
\end{equation}
It is clear that the algorithm is dominated by QR factorizations.

\paragraph{Global Block Arnoldi}
The global version of Algorithm~\ref{alg:block_arnoldi} can be made much cheaper than the classical one, because the global algorithm effectively reduces to running an Arnoldi routine on vectors of size $n^2p \times 1$ with the Frobenius norm and inner product.  Thus the QR factorization is replaced by $\bigO(n^3)$, the cost of normalizing a vector; the inner product is also now only $\bigO(n^3)$, and the $\vV C$ product reduces to multiplying $\vV$ by a scalar, which is again $\bigO(n^3)$.  The total cost of global Arnoldi is consequently dominated by $A \vV$ products:
\begin{equation} \label{O:gl_arnoldi}
	\cost_{\text{gl-Arnoldi}} = \bigO\left(m n^4 + \left(\frac{1}{2}(m+1)(m+2) \right)n^3 \right).
\end{equation}

\paragraph{Classical \bfomfom}
The computation of \eqref{eq:Fm} can be broken into three stages: $f(\HH_m)$, a $\vV C$-type product, and the evaluation of the basis $\bVV_m$ on the resulting matrix.  The matrix function should be computed via a direct algorithm, such as the Schur-Parlett Algorithm, described in \cite[Chapter 9]{Higham2008}.  Without function-specific information, $\texttt{funm}$ requires up to $\bigO(m^4n^4)$; function-specific algorithms, or algorithms that take advantage of the eigenvalue distribution of $A$ (described in other chapters of \cite{Higham2008}) may be cheaper.  The product $f(\HH_m) \vE_1$ does not require computation, since we can just extract the first block column from $f(\HH_m)$; $f(\HH_m) \vE_1 \Beta$ then requires only $\bigO(mn^3)$.  Finally, $\bVV_m$ applied to an $mn \times n$ matrix requires $\bigO(mn^4)$.  Including the Arnoldi cost \eqref{O:cl_arnoldi}, the total for computing \eqref{eq:Fm} is then
\begin{equation} \label{O:cl_bfomfom}
	\cost_{\text{cl-\bfomfom}} = \bigO\left((m+1)n^6 + \left(m^4 + m (m+3) \right)n^4 + mn^3\right).
\end{equation}

\paragraph{Global \bfomfom}
The same three stages apply for the global version of \bfomfom as for the classical, but we can make many computations cheaper.  The matrix function $f(\HH_m) = f(H_m) \otimes I_n$, where the Kronecker product need not be formed explicitly, so the cost reduces to $\bigO(m^4)$.  The matrix $\Beta$ can be regarded as a scalar, and using the same column-extracting trick, $f(\HH_m) \vE_1 \Beta$ comes at a negligible cost, $\bigO(m)$.  Finally, the product with the basis $\bVV_m$ can be reduced to taking scalar combinations of the basis vectors, amounting to $\bigO((m-1)n^3)$.  The total for \eqref{eq:Fm}, including the Arnoldi costs \eqref{O:gl_arnoldi} is
\begin{equation} \label{O:gl_bfomfom}
	\cost_{\text{gl-\bfomfom}} = \bigO\left(m n^4 + \left(\frac{1}{2}(m+1)(m+2) + m - 1\right)n^3 + m^4 \right).
\end{equation}

\paragraph{Restarts}
Determining the computational complexity for restarted \bfomfom is challenging, because the quadrature rule is adaptive, and the number of nodes per restart cycle plays a crucial role in how much work is done.  Typically, the cost per additional cycle should be less than computing the first step, and it should decrease as the algorithm approaches convergence, because the quadrature rule can be approximated progressively less accurately; see, e.g., \cite{FrommerGuettelSchweitzer2014a}.  In the worst-case scenario, however, the cost of successive cycles may be as expensive as the first, so it is reasonable to regard \eqref{O:cl_bfomfom} and \eqref{O:gl_bfomfom} as upper bounds.

\subsection{Block diagonalization and the discrete Fourier transform} \label{subsec:block_diag_dft}
Per recommendations in \cite{KilmerBramanHaoHoover2013, KilmerMartin2011}, we can improve the computational effort of $\fAABB$ by taking advantage of the fact that $\bcirc{\AA}$ can be block diagonalized by the discrete Fourier transform (DFT) along the tubal fibers of $\AA$.  Let $F_p$ denote the DFT of size $p \times p$.  Then we have that
\[
(F_p \otimes I_n) \bcirc{\AA} (F_p^* \otimes I_n)
= \begin{bmatrix}
D_1	&		&			&		\\
& D_2	&			&		\\
&		& \ddots	&		\\
&		&			& D_p
\end{bmatrix} =: D,
\]
where $D_k$ are $n \times n$ matrices.  Then by Theorem~\ref{thm:mat_func_props}(iii),
\[
f(\bcirc{\AA})
= (F_p^* \otimes I_n) f(D) (F_p \otimes I_n).
\]
Since each $D_i$, $i = 1, \ldots, p$ may be a full $n \times n$ matrix, applying $D$ itself to block vectors will still take $\bigO(n^4)$ operations.  However, this structure is easier to parallelize and requires less memory-movement than using $\AA$ directly, which will play an important role in high-performance applications.

\subsection{The tensor t-exponential on a small third-order network} \label{subsec:tensor_texp}
We take $\AA \in \spC^{n \times n \times p}$ to be a tensor whose $p$ frontal faces are each adjacency matrices for an undirected, unweighted network.  More specifically, the frontal faces of $\AA$ are symmetric, and the entries are binary.  The sparsity structure of this tensor is given in Figure~\ref{fig:spyAA} for $n = p = 50$. Note that we must actually compute $\exp(\AA) * \II = \fold{\exp(\bcirc{\AA})\vE_1}$ (\cf Definition~\eqref{def:fAA}).  With $n = p = 50$, this leads to a $2500 \times 2500$ matrix function times a $2500 \times 50$ block vector.  The sparsity patterns of $\bcirc{\AA}$ and $D$ are shown in Figure~\ref{fig:spy_plots}.  The block matrix $D$ is determined by applying \MATLAB's fast Fourier transform to $\bcirc{\AA}$.  Note that $\bcirc{\AA}$ is not symmetric, but it has a banded structure.  It should also be noted that while the blocks of $D$ appear to be structurally identical, they are not numerically equal.
\begin{figure}[htbp!]
	\begin{center}
		\includegraphics[width=.45\textwidth]{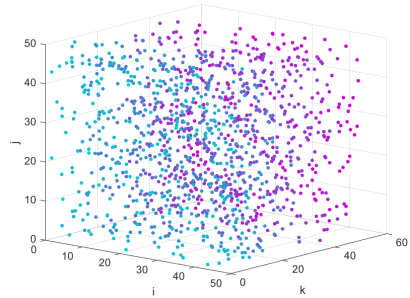}
	\end{center}
	\caption{\small Sparsity structure for $\AA$.  Blue indicates that a face is closer to the ``front" and pink farther to the ``back"; see Figure~\ref{fig:tensors}(f) for how the faces are oriented. \label{fig:spyAA}}
\end{figure}

\renewcommand{\imsize}{.25\textwidth}
\begin{figure}[htbp!]
	\begin{center}
		\begin{tabular}{c c c c c}			
			{\resizebox{\imsize}{!}{
					\includegraphics[width=\imsize]{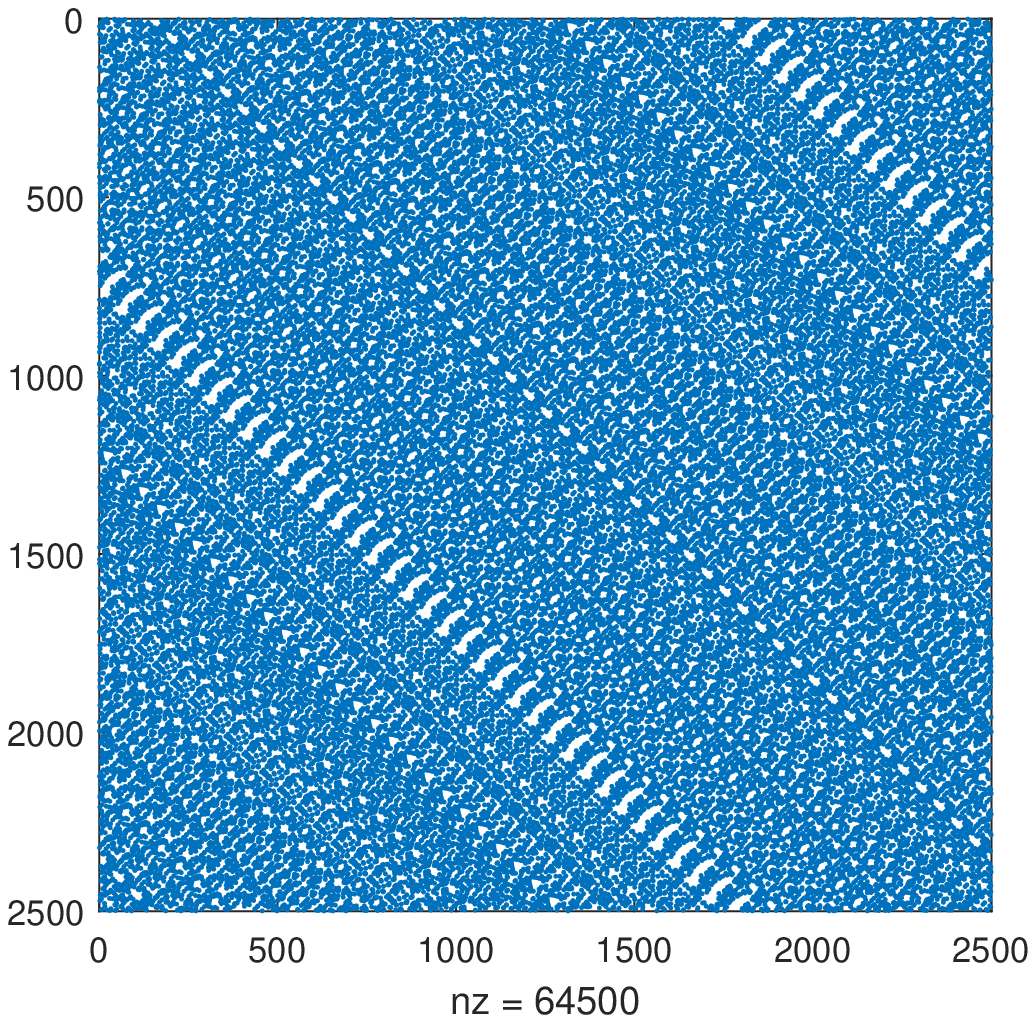}}} &
			
			&
			{\resizebox{\imsize}{!}{
					\includegraphics[width=\imsize]{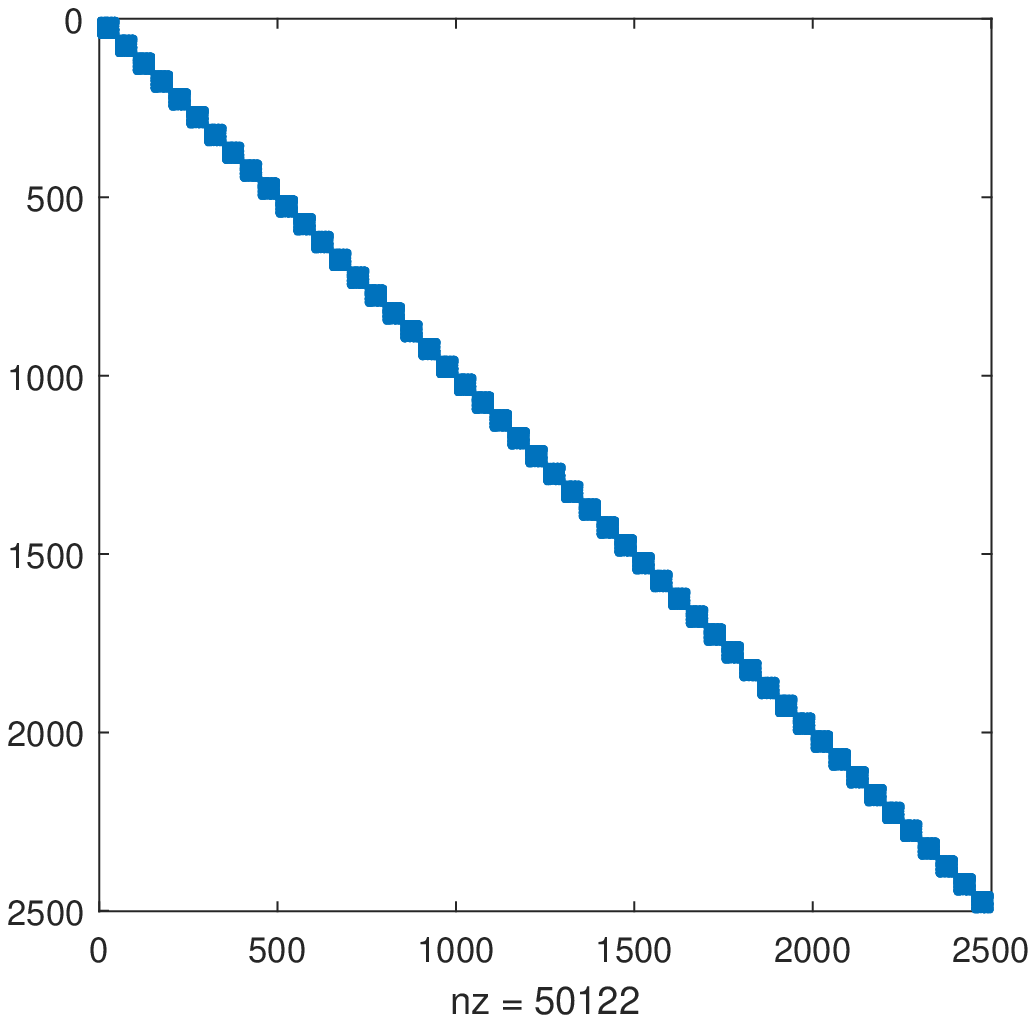}}} &
			&
			{\resizebox{\imsize}{!}{
					\includegraphics[width=\imsize]{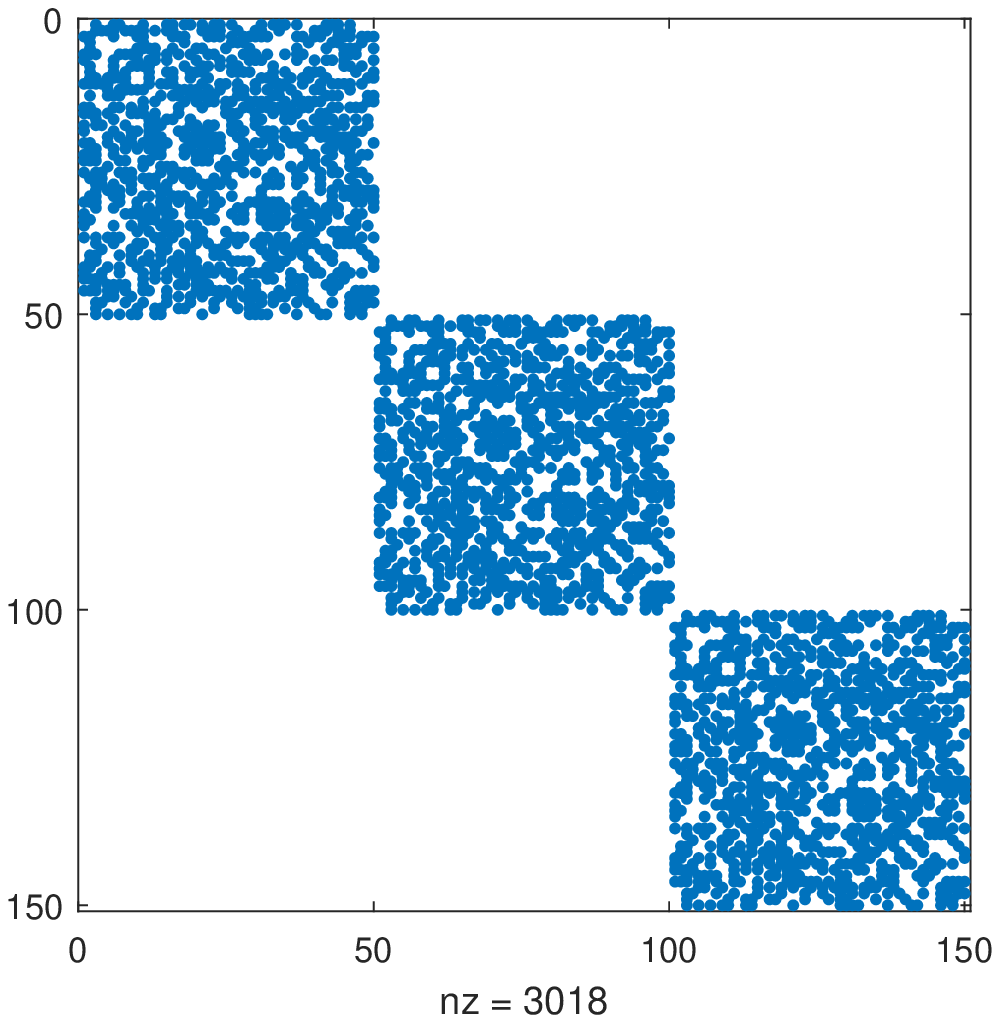}}} \\
			
			\small $\bcirc{\AA}$	& & \small $D$		& & \small zoom in on $D$
		\end{tabular}
	\end{center}
	\caption{\small Sparsity patterns for block circulants \label{fig:spy_plots}}
\end{figure}

We compute $\exp(\AA) * \II$ with the classical and global versions of \bfomfom using the \MATLAB software package \texttt{bfomfom}.\footnote{The script \texttt{tensor\_texp\_network\_v2.m} used to generate our results can be found at \url{https://gitlab.com/katlund/bfomfom-main}, along with the main body of code.} We run \MATLAB 2019a on Windows 10 on a laptop with 16GB RAM and an Intel i7 processor at 1.80GHz. The convergence behavior of each version is displayed in Figure~\ref{fig:conv_plots}, where we report the relative error per restart cycle, i.e., per $m$ iterations of Algorithm~\ref{alg:block_arnoldi}.  The restart cycle length is $m = 5$, and the error tolerance is $10^{-12}$.
\renewcommand{\imsize}{.42\textwidth}
\begin{figure}[htbp!]
	\begin{center}
		\begin{tabular}{c c}			
			{\resizebox{\imsize}{!}{
					\includegraphics[width=\imsize]{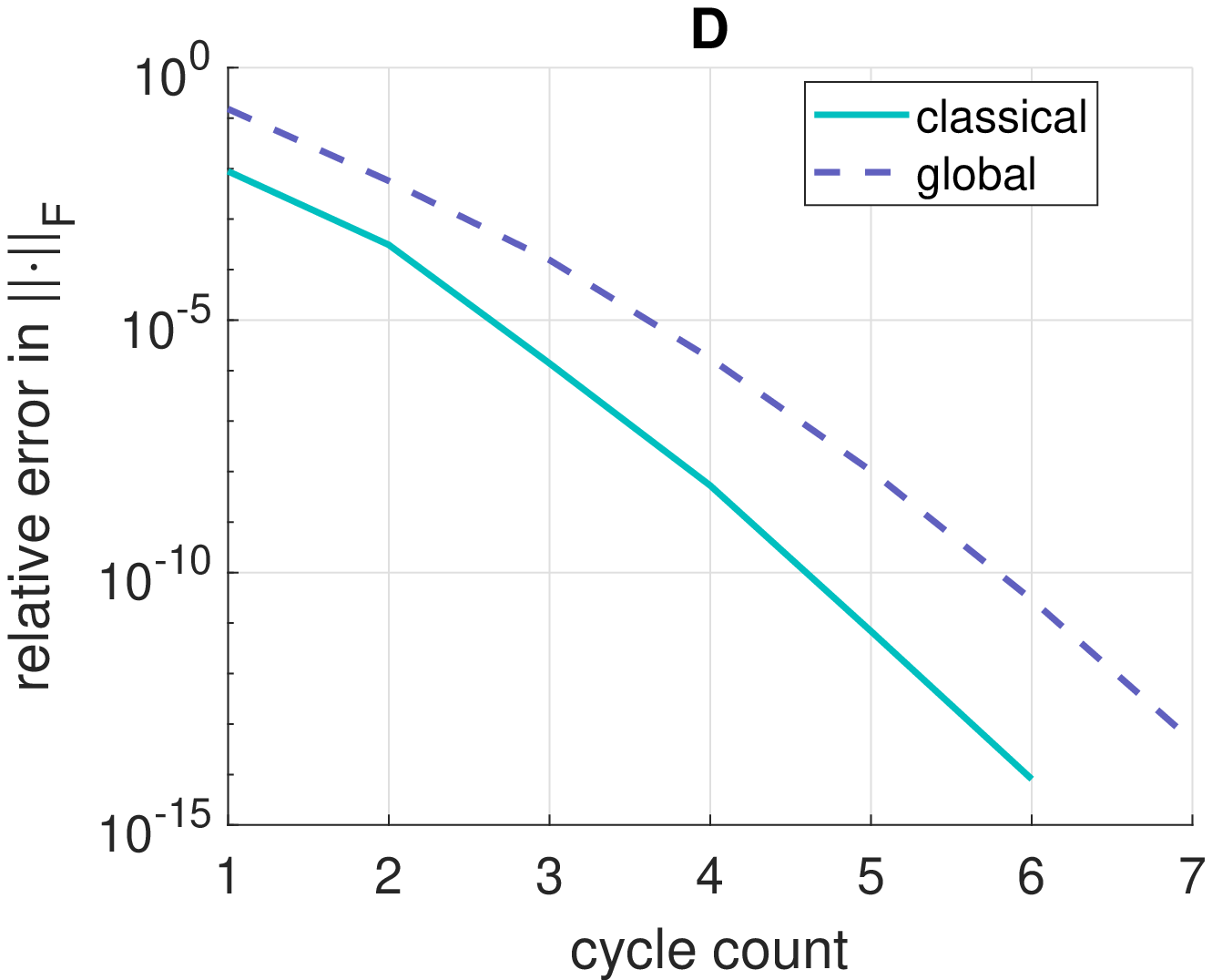}}} &
			
			{\resizebox{\imsize}{!}{
					\includegraphics[width=\imsize]{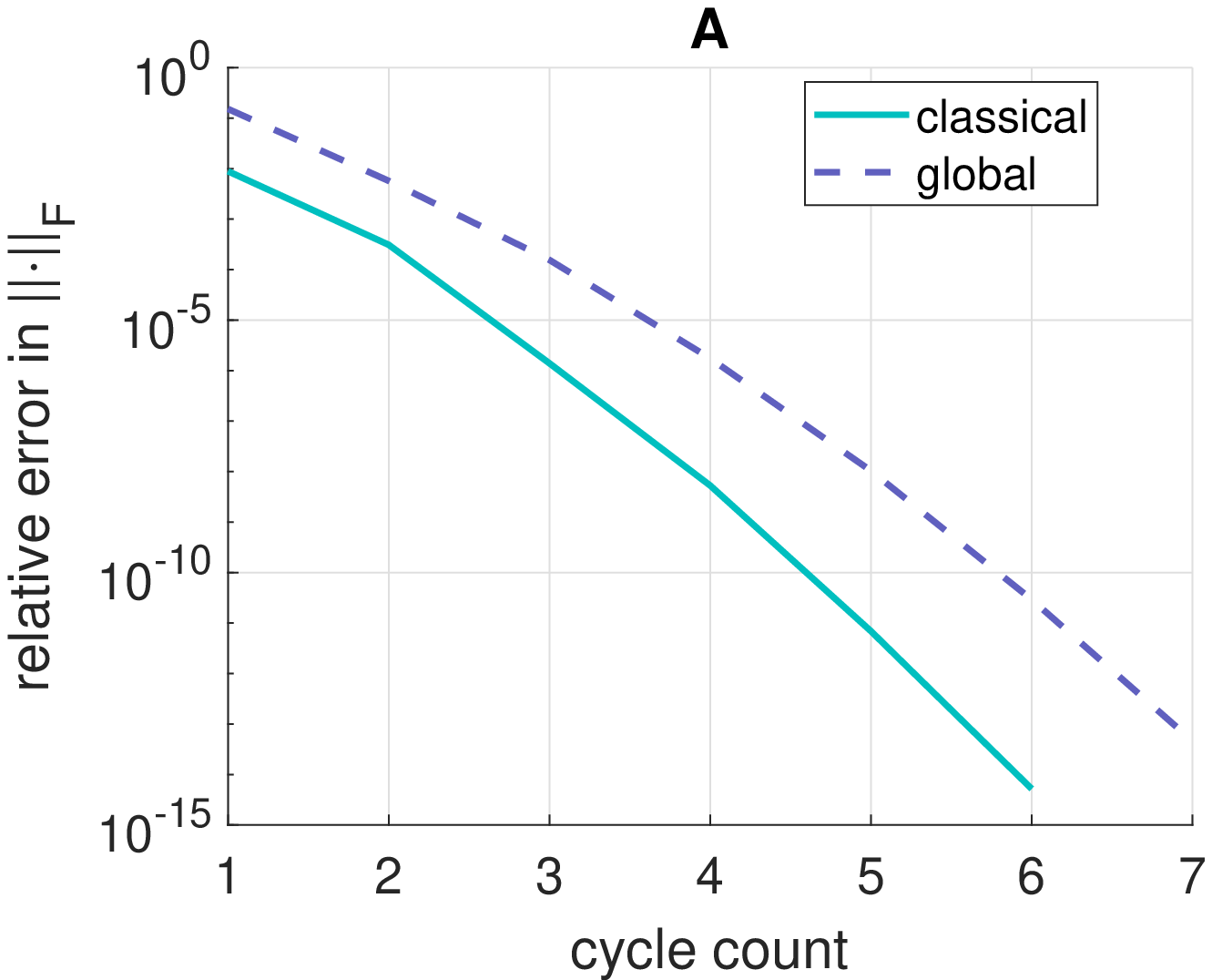}}} \\
			
			\small (A) & \small (B)
		\end{tabular}
	\end{center}
	\caption{\small Convergence plots for (A) classical and global methods on $\exp(D) F_p \otimes I_n \vE_1$, and (B) classical and global methods on $\exp(\bcirc{\AA}) \vE_1$. $m = 5$. \label{fig:conv_plots}}
\end{figure}
The methods based on $D$ (case (A)) are only a little less accurate than those based on $\bcirc{\AA}$ (case (B)), and they require the same number of iterations to converge.  The global methods require only one more cycle than the classical ones, but considering the computational complexity per cycle (cf.\ \eqref{O:cl_bfomfom} and \eqref{O:gl_bfomfom}), it is clear that the global methods require far less work overall.

Table~\ref{tab:vary_m} shows that for larger $m$, both classical and global methods require the same number of cycles (for either $D$- or $\bcirc{\AA}$-based approaches). For smaller values of $m$, global methods cannot attain the desired tolerance, because they exceed the maximum number of quadrature nodes allowed to perform the error update in line~\ref{line:errfunc} of Algorithm~\ref{alg:bfomfom_restarts}.  See, however, Figure~\ref{fig:conv_plots2} for the convergence behavior of the global method when $m=2$.  It still attains a high level of accuracy with much less work overall than the classical method.

\begin{table}[htbp!]
	\caption{Number of cycles needed to converge to $10^{-12}$ for different basis sizes $m$ \label{tab:vary_m}}
	\begin{center}
		\begin{tabular}{c|c|c|c|c}
					& $m=2$	& $m=5$	& $m=10$	& $m=15$	\\ \hline
		classical	& 18	& 6		& 3			& 2			\\ \hline
		global		& --	& 7		& 3			& 2			\\
		\end{tabular}
	\end{center}
\end{table}

\begin{figure}[htbp!]
	\begin{center}
		\begin{tabular}{c c}			
			{\resizebox{\imsize}{!}{
					\includegraphics[width=\imsize]{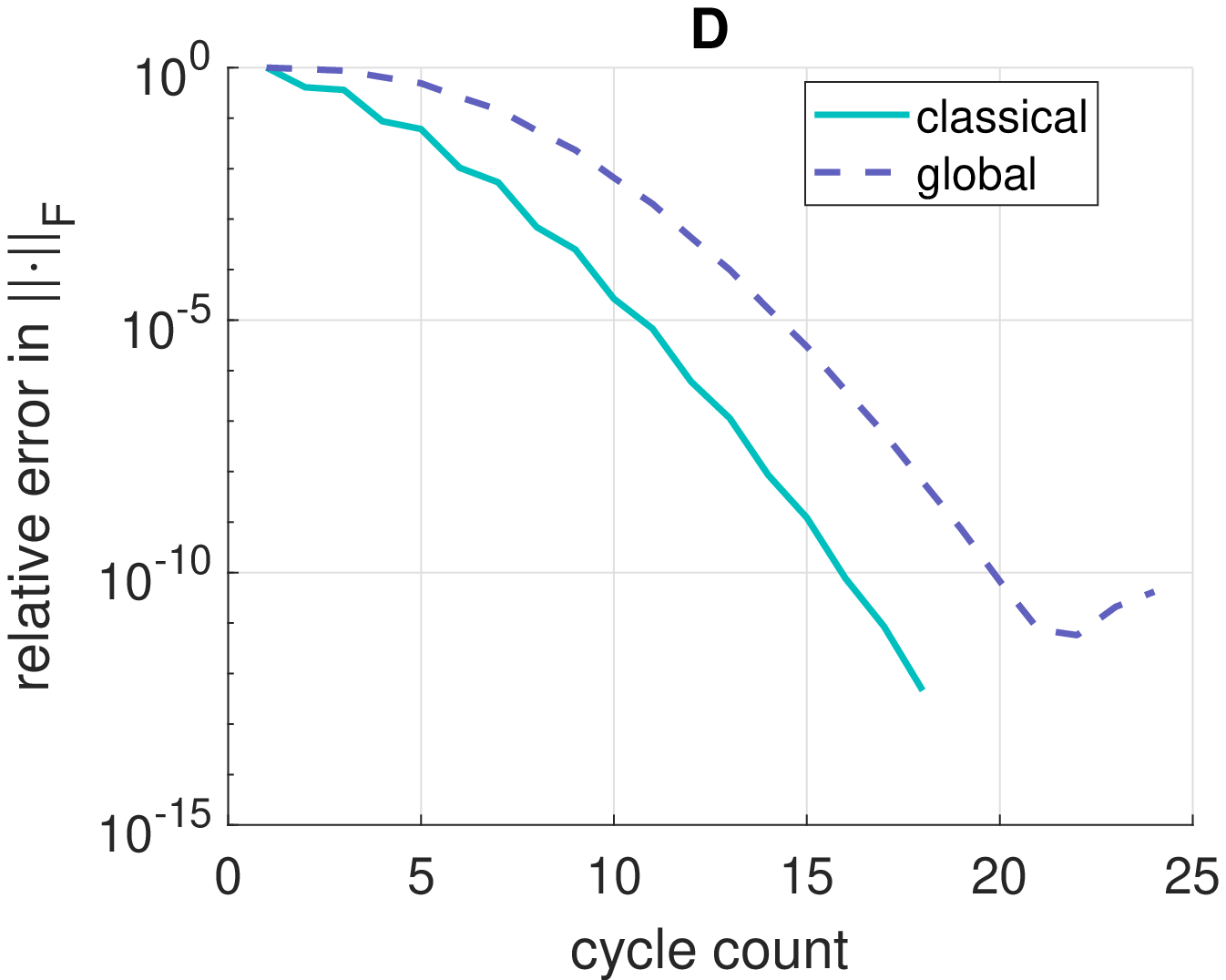}}} &
			
			{\resizebox{\imsize}{!}{
					\includegraphics[width=\imsize]{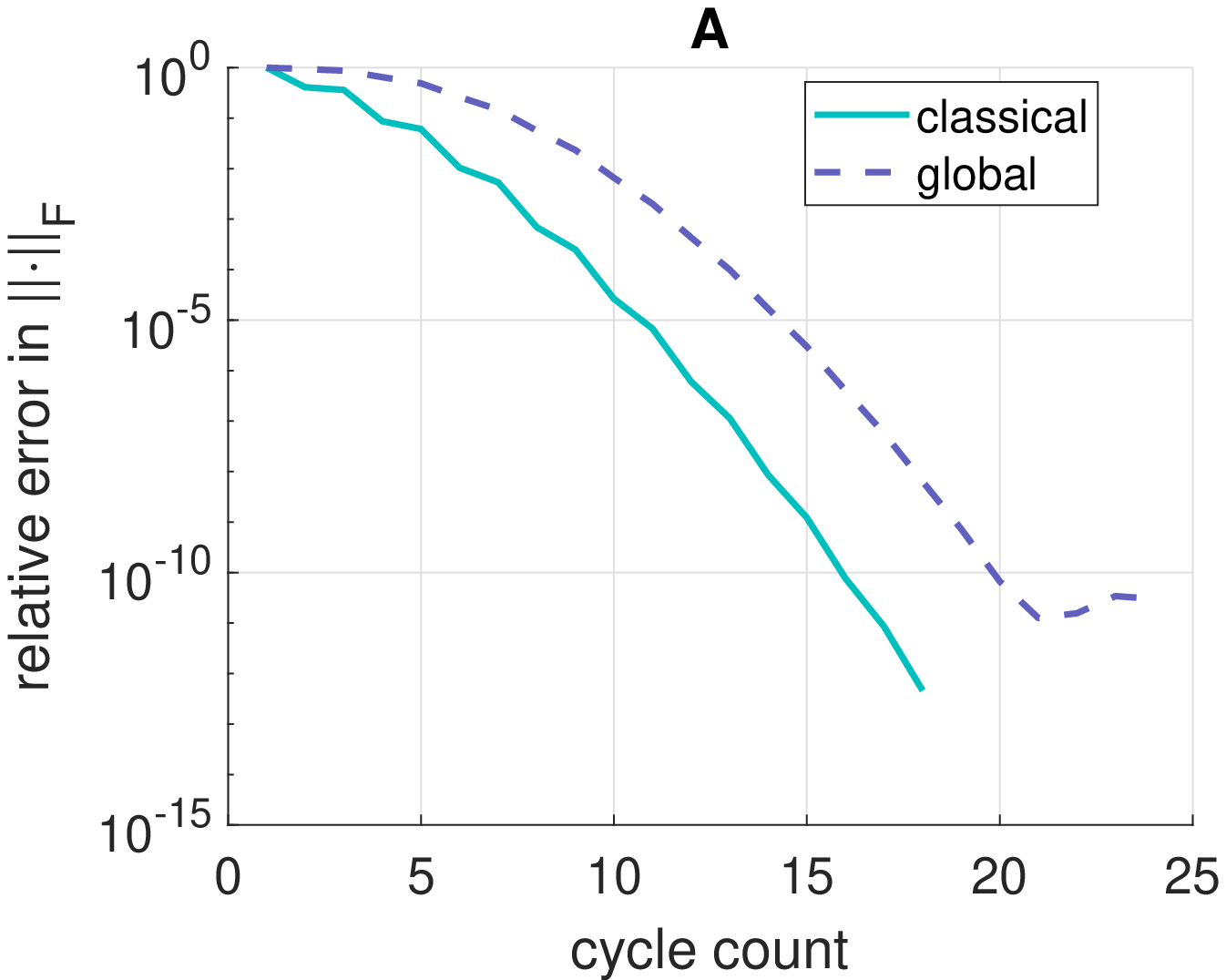}}} \\
			
			\small (A) & \small (B)
		\end{tabular}
	\end{center}
	\caption{\small Convergence plots for (A) classical and global methods on $\exp(D) F_p \otimes I_n \vE_1$, and (B) classical and global methods on $\exp(\bcirc{\AA}) \vE_1$. $m = 2$. \label{fig:conv_plots2}}
\end{figure}

\section{Conclusion} \label{sec:conclusions}
The main purpose of this report is to establish a first notion for functions of multidimensional arrays and demonstrate that it is feasible to compute this object with well understood tools from the matrix function literature.  Our definition for the tensor t-function $\fAABB$ shows versatility and consistency, and our numerical results indicate that block KSMs can compute $\fAABB$ with few iterations and still achieve high accuracy.  In particular, the global block KSM shows promise for moderate sizes, since its overall workload is significantly smaller than that of its classical counterpart.  For smaller basis sizes, which are more favorable in the context of large tensors, global methods may struggle to converge, and remedies for this situation remain an open problem.  One potential solution, that should first be explored for simple matrix functions, is to switch between global and classical paradigms in some optimal way so as to minimize overall computational effort while maximizing attainable accuracy.

The second aim of this report is to invite fellow researchers to pursue the many open problems posed by this new definition and to devise tensor function definitions for other paradigms.  Other key problems include exploring applications of $\fAABB$ in real-life scenarios and comparing our definition of communicability for a third-order network to existing network analysis tools.

\paragraph{Acknowledgments} 
The author would like to thank Misha Kilmer for useful conversations and the images used in Figure~\ref{fig:tensors}, Andreas Frommer and Daniel B. Szyld for comments on Theorem~\ref{thm:tensor_t-func_props}, Francesca Arrigo for suggesting a new application in Section~\ref{sec:network_analysis}, and the anonymous referee for multiple insights.

\bibliographystyle{siamplain}
\bibliography{tensors}

\end{document}